\documentclass[12pt]{article}
\usepackage{orcidlink}
\usepackage{amsmath}     
\usepackage{amssymb}     
\usepackage{amsthm}
\usepackage{amsfonts}
\usepackage{hyperref}
\usepackage[font=small]{caption}
\usepackage{tikz}
\usepackage[percent]{overpic}
\usepackage[showonlyrefs]{mathtools}
\mathtoolsset{showonlyrefs}
\numberwithin{equation}{section}
\def\re{\operatorname{Re}}
\def\im{\operatorname{Im}}
\def\arg{\operatorname{arg}}
\def\dist{\operatorname{dist}}
\def\sing{\operatorname*{sing}}
\def\interior{\operatorname*{int}}
\def\exterior{\operatorname*{ext}}
\def\N{\mathbb{N}}
\def\Z{\mathbb{Z}}
\def\R{\mathbb{R}}
\def\C{\mathbb{C}}
\def\H{\mathbb{H}}
\def\D{\mathbb{D}}

\def\Speiser{\mathcal{S}}
\def\EL{\mathcal{B}}
\newtheorem{theorem}{Theorem}[section]
\newtheorem{lemma}[theorem]{Lemma}
\newtheorem{proposition}{Proposition}[section]
\theoremstyle{definition}

\theoremstyle{remark}
\newtheorem{remark}[theorem]{Remark}
\newtheorem{example}[theorem]{Example}

\newtheorem*{ack}{Acknowledgment}
\begin{document}  
\title{On the 
boundary of an immediate attracting basin of a hyperbolic entire function}
\author{Walter Bergweiler\orcidlink{0000-0001-5345-1831} and Jie Ding\thanks{The second author gratefully acknowledges support by the CSC (202206935015) and the Fundamental Research Program of Shanxi Province (202103021224069).} \orcidlink{0000-0003-3023-3852}}
\date{}
\maketitle
\begin{abstract}
Let $f$ be a transcendental entire function of finite order 
which has an attracting periodic point $z_0$ of 
period at least~$2$. Suppose that the set of singularities of the inverse of $f$ is finite
and contained in the component $U$ of the Fatou set that contains $z_0$.
Under an additional hypothesis we show that the intersection of $\partial U$ with
the escaping set of $f$ has Hausdorff dimension~$1$.
The additional hypothesis is satisfied for example if $f$ has the form 
$f(z)=\int_0^z p(t)e^{q(t)}dt+c$ with polynomials $p$ and $q$ and  a constant~$c$.
This generalizes a result of Bara\'nski, Karpi\'nska and Zdunik dealing with the case
$f(z)=\lambda e^z$. 

\end{abstract}
\section{Introduction}\label{intro}
The \emph{Fatou set} $F(f)$ of an entire function $f$ is the set of all points in $\C$
where the iterates $f^n$ of $f$ form a normal family.
The \emph{Julia set} $J(f)$ is the complement of $F(f)$ and 
the \emph{escaping set} $I(f)$ is the set of all points $z$ in $\C$ for 
which $|f^n(z)|\to\infty$ as $n\to\infty$.
See \cite{Bergweiler1993,Schleicher2010} for an introduction to transcendental dynamics.

We say that $z_0\in\C$ is a \emph{periodic point} of period $p\in\N$
if $f^p(z_0)=z_0$, but $f^k(z_0)\neq z_0$ for $1\leq k\leq p-1$.
Of course, a periodic point of period $1$ is also called a \emph{fixed point}.
A periodic point $z_0$ of period $p$ is called \emph{attracting} if $|(f^p)'(z_0)|<1$.
In this case there exists a neighborhood $U$ of $z_0$ such that,
uniformly for $z\in U$, we have $f^{kp}(z)\to z_0$ as $k\to\infty$.
This easily yields that $z_0\in F(f)$. The set of all points $z$ such that 
$f^{kp}(z)\to z_0$ as $k\to\infty$ is called the \emph{basin of attraction} of $z_0$
and denoted by $A(z_0,f)$. 
The component of $A(z_0,f)$ that contains $z_0$ is called the \emph{immediate
basin of attraction} of $z_0$ and denoted by $A^*(z_0,f)$. 

The dynamics of the functions $E_\lambda(z):=\lambda e^z$ have been intensely studied;
see, e.g., \cite{Devaney2010,Rempe2006,Schleicher2003}.
If $E_\lambda$ has an attracting fixed point $z_0$, which is the case  for example if
$0<\lambda<1/e$, then $A^*(z_0,E_\lambda)=A(z_0,E_\lambda)=F(E_\lambda)$
and $\partial A^*(z_0,E_\lambda)=J(E_\lambda)$. 
McMullen~\cite[Theorem~1.2]{McMullen1987} showed that $\dim J(E_\lambda)=2$ for all $\lambda\in\C\setminus\{0\}$.
Here and in the following
$\dim X$ denotes the Hausdorff dimension of a set $X$.
We conclude that if $E_\lambda$ has an attracting fixed point $z_0$, then
$\dim\partial A^*(E_\lambda,z_0)=2$.

Bara\'nski, Karpi\'nska and Zdunik \cite{Baranski2010} showed that 
the last equality does not hold for attracting periodic points of higher period:
If $E_\lambda$ has an attracting periodic point $z_0$ of period $p\geq 2$, then
$\dim\partial A^*(z_0,E_\lambda)<2$.
One ingredient in their proof is a result of Urba\'nski and Zdunik~\cite[Theorem 6.1]{Urbanski2003}
which says that if $E_\lambda$ has an attracting periodic point, then 
$\dim (J(E_\lambda)\setminus I(E_\lambda))<2$.
Thus in order to prove that $\dim\partial A^*(z_0,E_\lambda)<2$,
Bara\'nski, Karpi\'nska and Zdunik only had to consider the dimension of
$\partial A^*(z_0,E_\lambda)\cap I(E_\lambda)$.
They showed \cite[Theorem~B]{Baranski2010} that 
if $z_0$ is an attracting periodic point of period at least $2$
of $E_\lambda$, then $\dim (A^*(z_0,E_\lambda)\cap I(E_\lambda))=1$.

The purpose of this paper is to extend  this result to a larger class of functions.
To define this class, recall that the set $\sing(f^{-1})$ of singularities of the
inverse of a transcendental entire function $f$ consists of the critical and (finite) asymptotic
values of~$f$. 
The \emph{Eremenko--Lyubich class} $\EL$ consists of all transcendental entire functions
$f$ for which $\sing(f^{-1})$ is bounded. The subclass $\Speiser$ of functions $f$
for which $\sing(f^{-1})$ is finite is called the \emph{Speiser class}.
These classes play an important role in transcendental dynamics;
see the survey by Sixsmith~\cite{Sixsmith2018}.

One reason why functions in these classes have been studied with great success
is the \emph{logarithmic change of variable}, introduced by Eremenko and 
Lyubich~\cite{Eremenko1992} to transcendental dynamics.
We describe it briefly. Here and in the following 
we will use, for $a\in\C$, $r>0$ and $t\in\R$, the notation 
\begin{equation} \label{e1}
D(a,r):=\{z\in\C\colon |z-a|<r\} 
\quad\text{and}\quad
\H_{>t} :=\{z\in\C\colon \re z>t\} .
\end{equation}
Of course, $\H_{\geq t}$, $\H_{<t}$ and $\H_{\leq t}$
are defined analogously.

For $f\in\EL$ we choose $R>\max\{1,|f(0)|\}$ such that $\sing(f^{-1})\subset D(0,R)$.
Put $s:=\log R$, $\Delta:=\{z\in\C\colon |z|>R\}$ and $T:=\exp^{-1}(f^{-1}(\Delta))$.
It can then be shown that there exists a holomorphic map $F\colon T\to \H_{>s}$ such that 
$f(e^z)=\exp F(z)$ for all $z\in T$.
Moreover, for every connected component $L$ of $T$ the map
$F\colon L\to \H_{>s}$ is biholomorphic.
This allows to apply results about univalent functions to $F$.
A particularly useful result is the estimate
\begin{equation} \label{e2}
|F'(w)|\geq \frac{1}{4\pi}(\re F(w)-s)
\quad\text{for}\ w\in T,
\end{equation}
which can be obtained by applying the Koebe one quarter theorem to the inverse
of $F\colon L\to \H_{>s}$.

The function $F$ is called the \emph{logarithmic transform} of $f$.
We also say that $F$ is obtained from $f$ by a logarithmic change of variable.
We will explain this logarithmic change of variable in more detail in \S\ref{log-change}.
The reason for this more detailed discussion is that 
we need to extend the domain of definition of $F$ to a larger set than above. Essentially, we 
want to achieve that $\exp T$ contains $J(f)$ so that we can define the Julia set
of $F$ by $J(F)=\exp^{-1} J(f)$.

It follows from~\eqref{e2} that if $\re F(w)\geq 2s$, then
$|F'(w)|\geq \re F(w)/(8\pi)$. 
For our method we will need an estimate of $|F'(w)|$ in terms 
of $|F(w)|$ rather than $\re F(w)$. 
We will thus make the additional hypothesis that 
there exist $\beta>0$ and $t>0$ such that 
\begin{equation} \label{b2}
|F'(w)|\geq \beta |F(w)|
\quad \text{if}\ \re F(w)\geq t.
\end{equation}
To rewrite these conditions in terms of $f$,
recall that
$F(w)=\log f(e^w)$
with some branch of the logarithm.
Thus
\begin{equation} \label{v0a}
F'(w)=\frac{e^wf'(e^w)}{f(e^w)} .
\end{equation}
With $z=e^w$ the inequalities~\eqref{e2} and~\eqref{b2} hence take the form
\begin{equation} \label{e2a}
\left|\frac{zf'(z)}{f(z)}\right|
\geq \frac{1}{4\pi}(\log |f(z)|-s)
\quad\text{if}\ |f(z)|>R=e^s
\end{equation}
and
\begin{equation} \label{b2a}
\left|\frac{zf'(z)}{f(z)}\right|
\geq \beta |\log f(z)|
\quad \text{if}\ |f(z)|\geq e^t.
\end{equation}
Thus our additional hypothesis involves not only the modulus but also the
argument of $f(z)$.

We will see in Proposition~\ref{prop1} that this hypothesis is satisfied for example if $f$ has 
the form 
\begin{equation} \label{b3}
f(z)=\int_0^z p(t)e^{q(t)}dt+c
\end{equation}
with polynomials $p$ and $q$ and $c\in\C$.
It is also easy to see that~\eqref{b2} holds for functions $f$ of
 the form $f(z)=a\cos z+b$ with $a,b\in\C$, $a\neq 0$.

An entire function $f$ is said to be of \emph{finite order} if there exists $\mu>0$ such that 
\begin{equation} \label{e4}
|f(z)|\leq \exp|z|^{\mu}
\end{equation}
if $|z|$ is sufficiently large.
The infimum of all these $\mu$ is called the \emph{order} of $f$ and denoted by $\rho(f)$.

Besides the escaping set $I(f)$ we will, for $M>0$, also consider the set
\begin{equation} \label{b0b}
I(f,M):=\left\{z \in {\mathbb C}\colon \liminf_{n\to \infty}|f^n(z)|\geq M\right\}.
\end{equation}
Note that
\begin{equation} \label{b0c}
  I(f)=\bigcap_{M>0}I(f,M).
\end{equation}

\begin{theorem}\label{thm1}
Let $f\in\Speiser$ be of finite order and suppose that the logarithmic transform $F$
of $f$ satisfies~\eqref{b2}.
Suppose that $f$ has an attracting periodic point $z_0$ of period $p\geq 2$.
For $1\leq j\leq p-1$, put $z_j:=f^j(z_0)$.
Suppose that there exists $j\in\{0,\dots,p-1\}$ such that  $\sing(f^{-1})\subset A^*(z_j,f)$.
Then
\begin{equation} \label{b4}
\lim_{M\to\infty} 
\dim (\partial A^*(z_j,f)\cap I(f,M)) =1 
\end{equation}
for all $j\in\{0,\dots,p-1\}$.
\end{theorem}

Of course, it follows from \eqref{b0c} and \eqref{b4} that
\begin{equation} \label{b5}
\dim (\partial A^*(z_j,f)\cap I(f)) =1 .
\end{equation}

A function $f\in \EL$ for which the closure of $\sing(f^{-1})$ is contained in
attracting basins is called \emph{hyperbolic}.
If $f$ is hyperbolic, then $F(f)$ consists of finitely many attracting basins;
see, e.g., \cite[Proposition~2.1]{Bergweiler2015}.
A standard result in complex dynamics \cite[Theorem~7]{Bergweiler1993} says that 
a periodic cycle of immediate attracting basins contains a singularity of the inverse.
We conclude that if $f$ satisfies the hypotheses of Theorem~\ref{thm1},
then $F(f)=\bigcup_{j=0}^{p-1} A(z_j,f)$. 
Thus every component of $F(f)$ is mapped to $A^*(z_0,f)$ by some iterate of~$f$.
This implies that $\dim (\partial U\cap I(f))=1$ for every component $U$ of $F(f)$.

We just mentioned the result that if an entire function $f$ has an attracting periodic
point $z_0$ of period $p$, and we put $z_j=f^j(z_0)$ as above, then 
\begin{equation} \label{b6}
\sing(f^{-1})\cap \bigcup_{j=0}^{p-1} A^*(z_j,f)\neq\emptyset . 
\end{equation}
The essential hypothesis in Theorem~\ref{thm1} is that there exists 
$j\in\{0,\dots,p-1\}$ such that  $\sing(f^{-1})\subset A^*(z_j,f)$.
We will show by an example that without this hypothesis the conclusion need not hold.
Let
\begin{equation} \label{b7}
f_\lambda(z):=\lambda\int_0^z \exp(-t^2) dt,
\end{equation}
with $\lambda\in\R\setminus\{0\}$.
\begin{example} \label{ex1}
The  function $f_{-2}$ has an attracting periodic 
point $z_0\approx 1.7487\dots$ of period~$2$, with $z_1:=f_{-2}(z_0)=-z_0$,
such that 
\begin{equation} \label{b8}
\dim(\partial A^*(z_j,f_{-2})\cap I(f_{-2}))=2
\end{equation}
for $j\in\{0,1\}$.
Moreover, $\sing(f_{-2}^{-1})=\{\pm \sqrt{\pi}\}$, with $\sqrt{\pi}\in A^*(z_0,f_{-2})$ and
$-\sqrt{\pi}\in A^*(z_1,f_{-2})$.
\end{example} 

The left picture in Figure~\ref{fig1} shows the attracting basins of the function
$f_{-2}$ from Example~\ref{ex1}.
\begin{figure}[htb]
\centering
\begin{overpic}[width=0.47\textwidth]{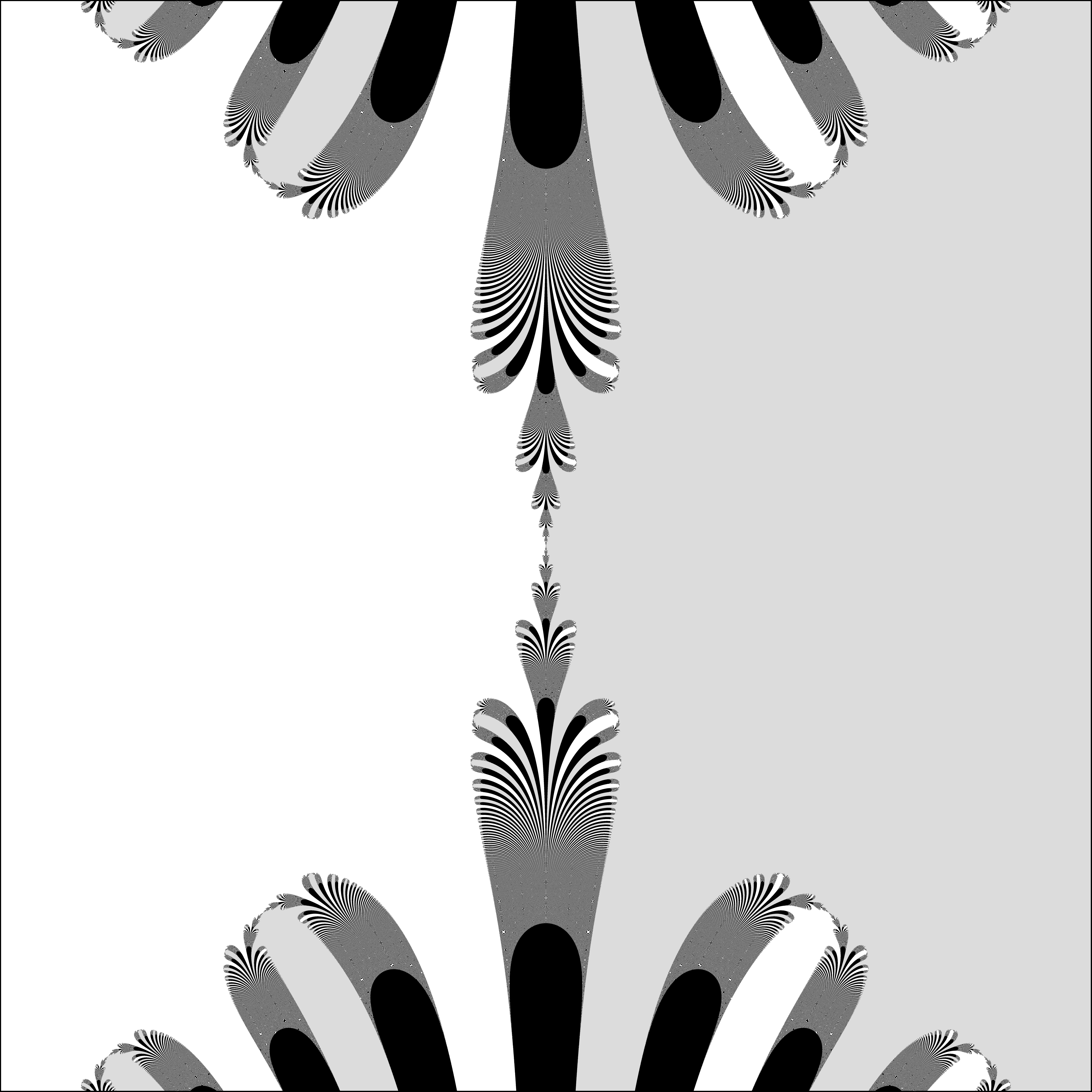}
\put(19.11,50){\circle*{1}}
\put(20.11,47){$z_1$}
\put(80.89,50){\circle*{1}}
\put(81.89,47){$z_0$}
\end{overpic}
\quad
\begin{overpic}[width=0.47\textwidth]{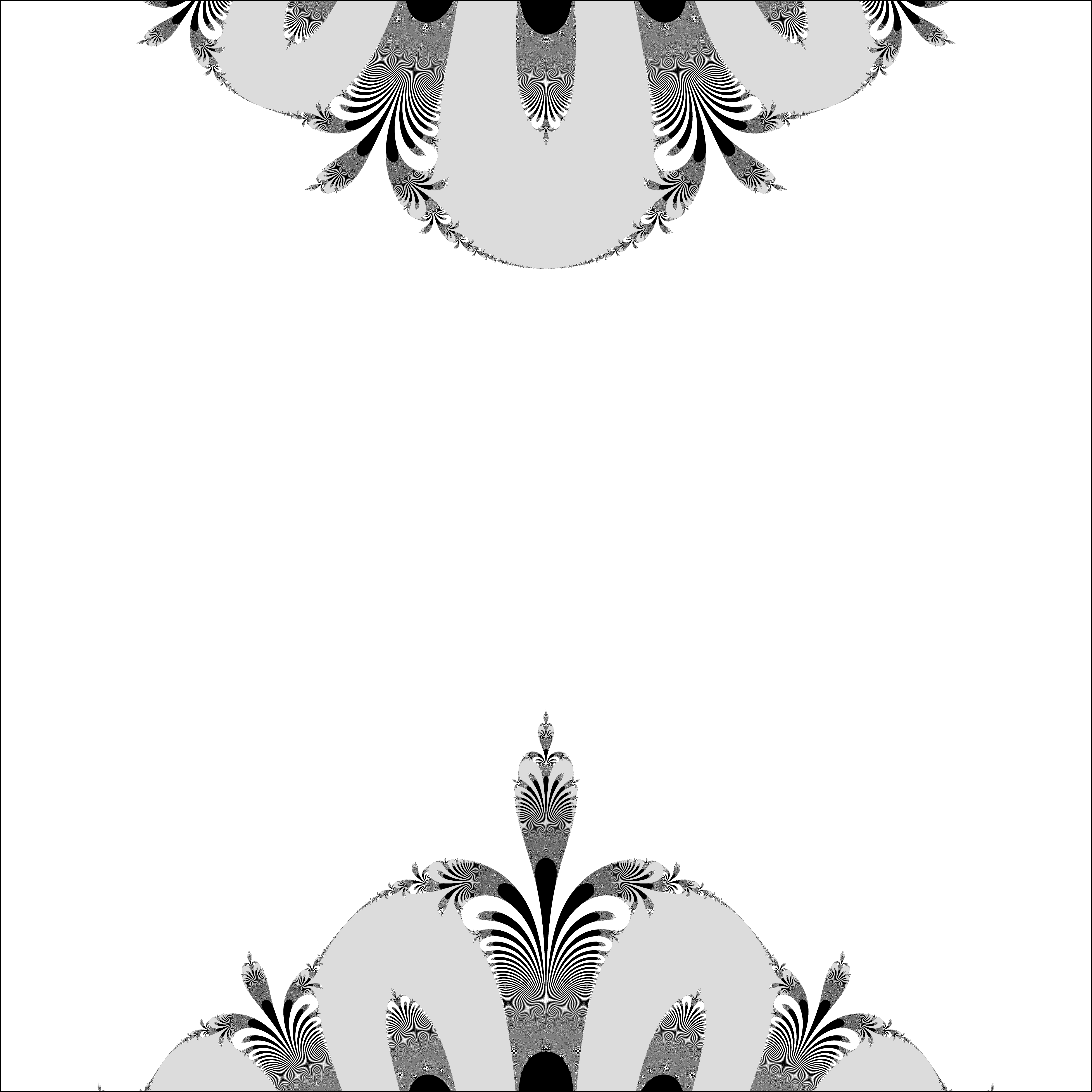}
\put(50,64.05){\circle*{1}}
\put(51,61.05){$z_0$}
\put(50,81.47){\circle*{1}}
\put(51,78.47){$z_1$}
\end{overpic}
\caption{Attracting basins for the function from Example \ref{ex1}.}
\label{fig1}
\end{figure}
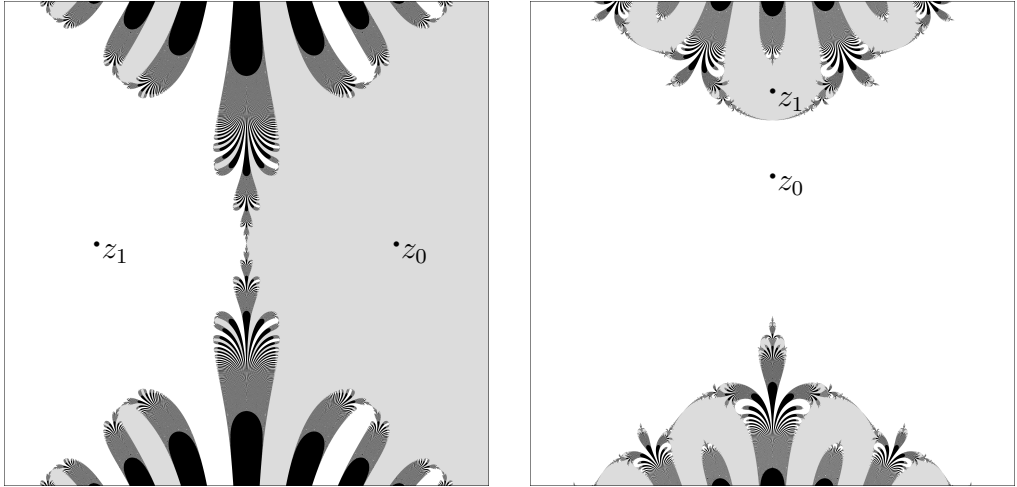
The right picture shows the attracting basins of $f:=f_{-0.14}+1.9i$, which has the attracting
periodic point $z_0\approx 0.7868 i$ of period~$2$, with $z_1:=f(z_0)\approx 1.7621 i$.
This example was considered by Morosawa~\cite{Morosawa2004}; see \S\ref{veri} for a
discussion of his work.
For this function $f$ we have 
$\sing(f^{-1})=\{1.9i\pm 0.07\sqrt{\pi}\}\subset A^*(z_1,f)$
so that the hypotheses of Theorem~\ref{thm1} are satisfied.
In both pictures the range shown is $|\re z|\leq 2.8$, $|\im z|\leq 2.8$.

Figure~\ref{fig2} shows attracting basins for functions $f$ of the form
$f(z)=a\cos z+b$,
with $a$ and $b$ chosen such that the hypotheses of Theorem~\ref{thm1} are satisfied.
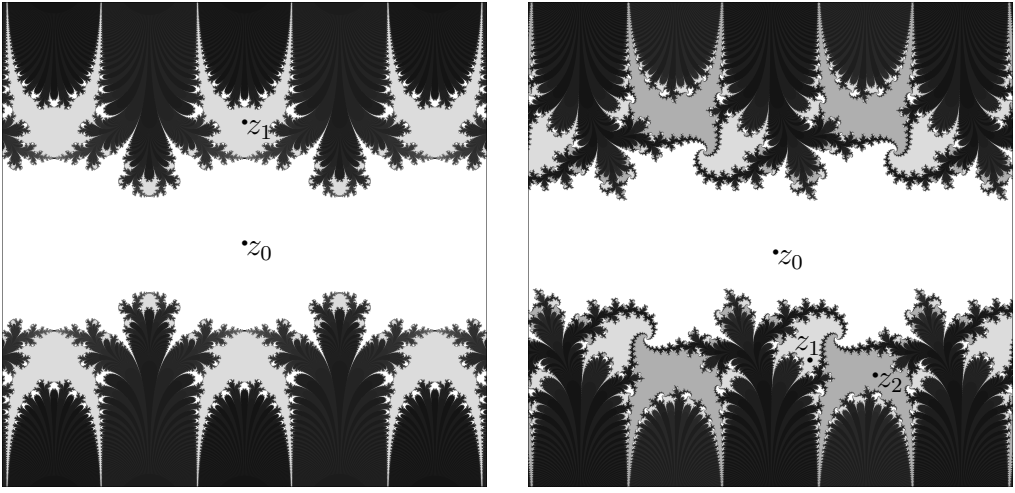
\begin{figure}[htb]
\centering
\begin{overpic}[width=0.47\textwidth]{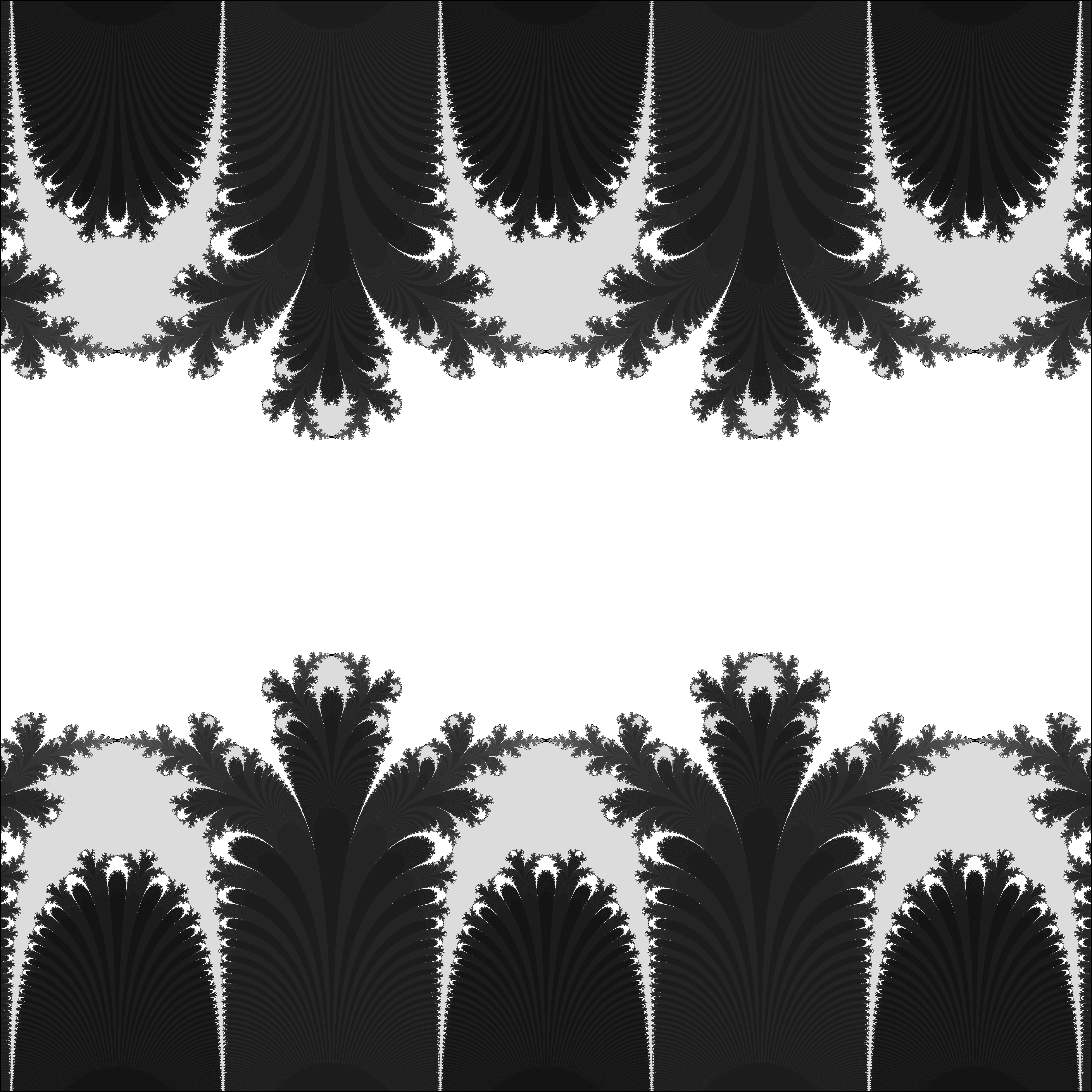}
\put(50,50.34){\circle*{1}}
\put(50.3,47.84){$z_0$}
\put(50,75.31){\circle*{1}}
\put(50.3,72.81){$z_1$}
\end{overpic}
\quad
\begin{overpic}[width=0.47\textwidth]{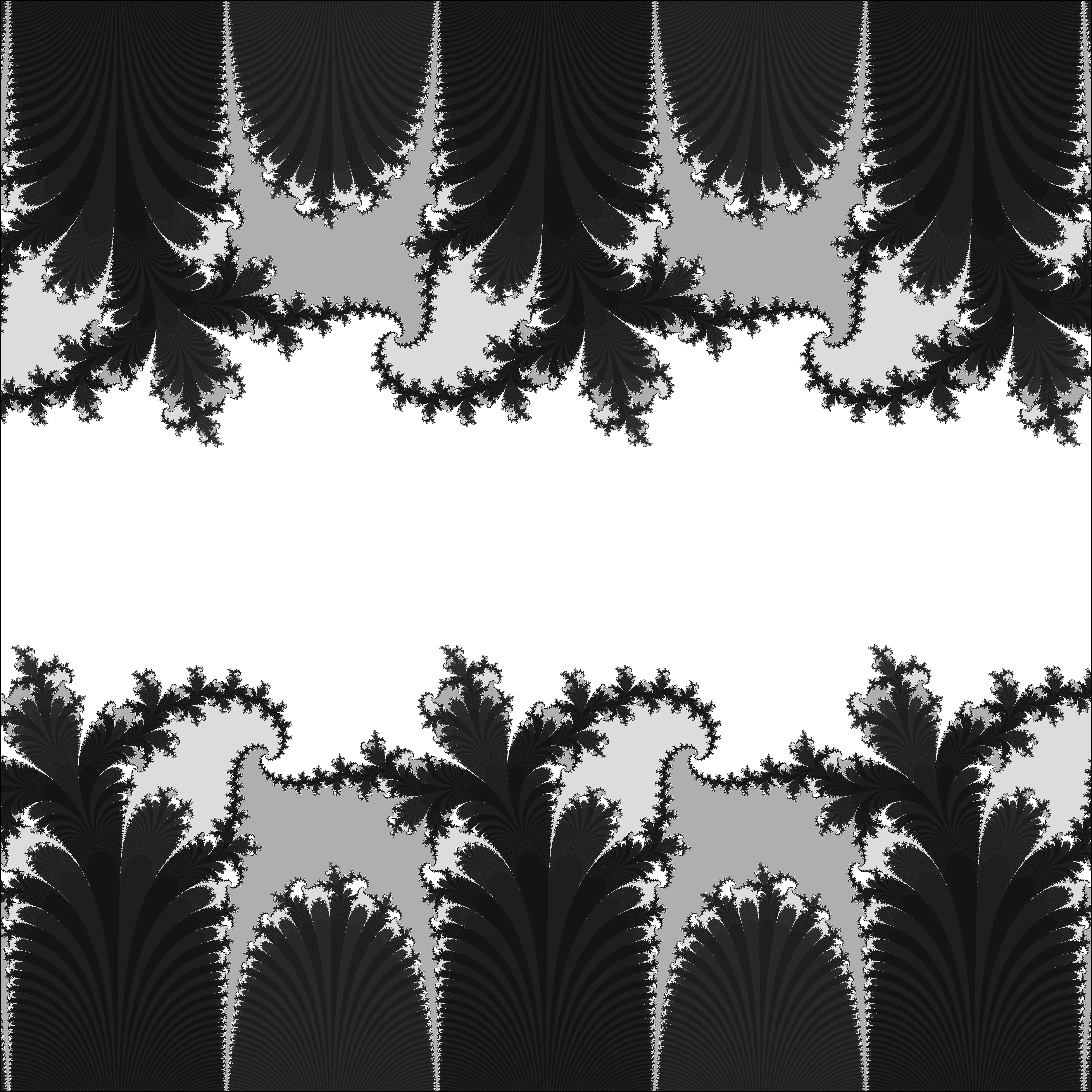}
\put(51.04,48.57){\circle*{1}}
\put(51.34,46.07){$z_0$}
\put(58.148,26.242){\circle*{1}}
\put(54.825,28.25){$z_1$}
\put(71.6,23.18){\circle*{1}}
\put(71.9,20.68){$z_2$}
\end{overpic}
\caption{Attracting basins of some cosine maps.}
\label{fig2}
\end{figure}
In the left picture we have $a=-0.15i$ and $b=4.15 i$. 
Here $f$ has an attracting periodic point $z_0\approx -0.05463 i$ of period~$2$.
In the right picture, $a=-0.1i$ and $b=1.3-3.7i$, and 
$f$ has an attracting periodic point $z_0\approx 0.1662-0.2292 i$ of period~$3$.
In both cases, 
\begin{equation} \label{b9}
\sing(f^{-1})=\{b+a,b-a\}\subset A^*(z_1,f)
\end{equation}
so that the hypotheses of Theorem~\ref{thm1} are satisfied.
In both pictures the range shown is $|\re z|\leq 8$, $|\im z|\leq 8$.

It follows from~\cite[Theorem 1.4]{Bergweiler2015} that if $f(z)=a\cos z+b$ is hyperbolic and
the two critical values $b\pm a$ are in different components of $F(f)$, then all 
components of $F(f)$ are bounded quasidisks. 
In particular, the boundary of an immediate attracting basin does not intersect the 
escaping set.

We do not know whether under the hypotheses of Theorem~\ref{thm1} we have
$\dim\partial A^*(z_j,f)<2$.
As we already mentioned, this was proved by 
Bara\'nski, Karpi\'nska and Zdunik \cite{Baranski2010} for the case that $f=E_\lambda$,
using the work of Urba\'nski and Zdunik~\cite{Urbanski2003}.

We remark that if $f$ is hyperbolic and $M>0$, then
\begin{equation} \label{i1}
\dim\!\left\{z\in J(f) \colon \limsup_{n\to\infty} |f^n(z)|\leq  M \right\} <2.
\end{equation}
To see this, we note that one way to prove that the Julia set of a hyperbolic rational 
function has Hausdorff dimension less than $2$ is to show that it is porous.
See the book by Berteloot and Mayer~\cite[Section~VI.3]{Berteloot2001} for the 
details of this argument.
Essentially the same reasoning shows that the set occurring in~\eqref{i1} is 
porous and thus has Hausdorff dimension less than~$2$,
for each $M>0$. However, the estimate obtained is not uniform in $M$,
and the upper bound will usually tend to $2$ as $M\to\infty$.

But even if there was a uniform bound in~\eqref{i1} independent of~$M$, this 
would not yield that $\dim\partial A^*(z_j,f)<2$.
To achieve this one would also have to estimate the dimension of the 
intersection of $\partial A^*(z_j,f)$ with the so-called \emph{bungee set}
consisting of the points $z$ for which 
$\liminf_{n\to\infty} |f^n(z)|<\infty$ and
$\limsup_{n\to\infty} |f^n(z)|=\infty$.

We also do not know whether the hypothesis \eqref{b2} is essential.

The basic idea in the proof of Theorem~\ref{thm1} is the same as in~\cite{Baranski2010}:
Points in the boundary of the attracting basin have ``itineraries'' of a particular form.
Here we use this approach for the logarithmic transform.
Thus, after some preliminaries in \S\ref{hyp-geom}, we introduce the logarithmic 
transform in \S\ref{log-change} and consider attracting basins and 
 itineraries for it in \S\ref{att-basin}. 
In \S\ref{geom-tract} we discuss some results on the geometry of the domains where the logarithmic 
transform is defined and in \S\ref{quasi} we introduce a quasiconformal modification of~$f$.
The proof of the upper bound for the dimension is then given in \S\ref{hausdimI}. 
The much simpler proof of the lower bound follows in \S\ref{hausdimI2}.
As already mentioned, condition~\eqref{b2} is discussed in \S\ref{hypo-thm1}.
Finally, we verify in \S\ref{veri} that the function from Example~\ref{ex1} has the properties
stated.
\begin{ack}
We thank the referee for a number of very valuable comments and suggestions.
\end{ack}

\section{Preliminaries}\label{hyp-geom}
The Koebe one quarter theorem is 
usually stated for functions univalent in the unit disk.
The following version for functions univalent in an arbitrary disk 
 is deduced by a simple transformation from this.
\begin{lemma}\label{koebe}
Let $f\colon D(a,r)\to\C$ be univalent.  Then
\begin{equation} \label{koebe3}
f(D(a,r))\supset D\!\left(f(a),\frac14 r|f'(a)|\right).
\end{equation}
\end{lemma}

We will use some standard results from hyperbolic geometry; see \cite{Beardon2007}
for an exposition of all results needed.
For a hyperbolic domain $\Omega$ let $\lambda_\Omega(z)$ be the 
\emph{density of the hyperbolic metric} in~$\Omega$. 
For the unit disk $\D$ we thus have 
\begin{equation} \label{c4a}
\lambda_\D(z)=\frac{2|z|}{1-|z|^2}.
\end{equation}
The \emph{hyperbolic length} $\ell_\Omega(\gamma)$ of a curve $\gamma$ is given by
\begin{equation} \label{c4b}
\ell_\Omega(\gamma) =\int_\gamma \lambda_\Omega(z)|dz| .
\end{equation}
The Euclidean length is denoted by $\ell(\gamma)$.

The Schwarz--Pick lemma says that if $U$ and $V$ are hyperbolic domains
and $f\colon U\to V$ is holomorphic, then
\begin{equation} \label{c5c}
\lambda_V(f(z))|f'(z)|\leq \lambda_U(z)
\end{equation}
for all $z\in U$, with equality if $f$ is a covering.
Conversely, if equality in~\eqref{c5c} holds for one $z\in U$,
then $f$ is a covering (and thus equality holds for all $z\in U$).

If $U\subset V$, we may apply~\eqref{c5c} with $f(z)=z$ and find that 
\begin{equation} \label{c5d}
\lambda_V(z)\leq \lambda_U(z),
\end{equation}
with strict inequality if $U$ is a proper subset of~$V$.

The Koebe one quarter theorem \eqref{koebe3} yields that
for a simply connected domain $\Omega$ we have 
\begin{equation} \label{c5b}
\frac{1}{2\dist(z,\partial\Omega)} \leq \lambda_\Omega(z)
\leq \frac{2}{\dist(z,\partial\Omega)} 
\end{equation}
for $z\in\Omega$.
Here $\dist(z,\partial\Omega)$ is the Euclidean distance of $z$ to $\partial\Omega$.

For a half-plane $H$ we have
\begin{equation} \label{c5e}
\lambda_H(z) = \frac{1}{\dist(z,\partial H)} .
\end{equation}

\section{The logarithmic change of variable}\label{log-change}
We describe the logarithmic change of variable in more detail than in the introduction,
and refer to \cite[\S 2]{Eremenko1992} or \cite[\S 5]{Sixsmith2018}
for additional information.  Let $f$ be as in Theorem~\ref{thm1}.
Without loss of generality we may assume that $\sing(f^{-1})\subset A^*(z_1,f)$.
We may also assume that $0\in A^*(z_0,f)\setminus\{z_0\}$ and $f(0)\in A^*(z_1,f)\setminus\{z_1\}$, 
as this can be achieved by conjugation with $z\mapsto z+c$ for a suitable $c\in\C$.

We will also assume that $z_1\notin \sing(f^{-1})$.
As we shall see in \S\ref{quasi}, the general case can be reduced to this case.
Then there exists a Jordan curve $\gamma$ in $A^*(z_1,f)$ such that 
\begin{equation} \label{l1}
\sing(f^{-1})\cup \{f(0)\}\subset \interior(\gamma)
\end{equation}
and
\begin{equation} \label{l2}
z_1\in\exterior(\gamma). 
\end{equation}
Here $\interior(\gamma)$ and $\exterior(\gamma)$ denote the interior and exterior 
of $\gamma$, respectively.
Since $\gamma$ is contained in $A^*(z_1,f)$ but $0\in A^*(z_0,f)$ and hence $0\notin A^*(z_1,f)$
 we have $0\in\exterior(\gamma)$.
We connect $\gamma$ to $0$ by a curve $\gamma_0$ which except for its starting
point on $\gamma$ is contained in $\exterior(\gamma)$ and which does not contain any 
of the points~$z_j$. Let 
\begin{equation} \label{l3}
W:=\exterior(\gamma)\setminus\gamma_0,\  
V:=f^{-1}(W),\  
U:=\exp^{-1}(W)\  
\text{and} \ 
T:=\exp^{-1}(V).
\end{equation}

There exists a map $\phi\colon V\to U$ such that the restriction of $\phi$ to a component 
of $V$ is a biholomorphic map from this component onto $U$. 
Thus $F\colon T\to U$, $F(w)=\phi(e^w)$, is a $2\pi i$-periodic map with the property
that the restriction of $F$ to a component of $T$ is a biholomorphic map from
 this component onto $U$. We also put $\Gamma_0:=\exp^{-1}(\gamma_0)$.
See Figure~\ref{logchange}.
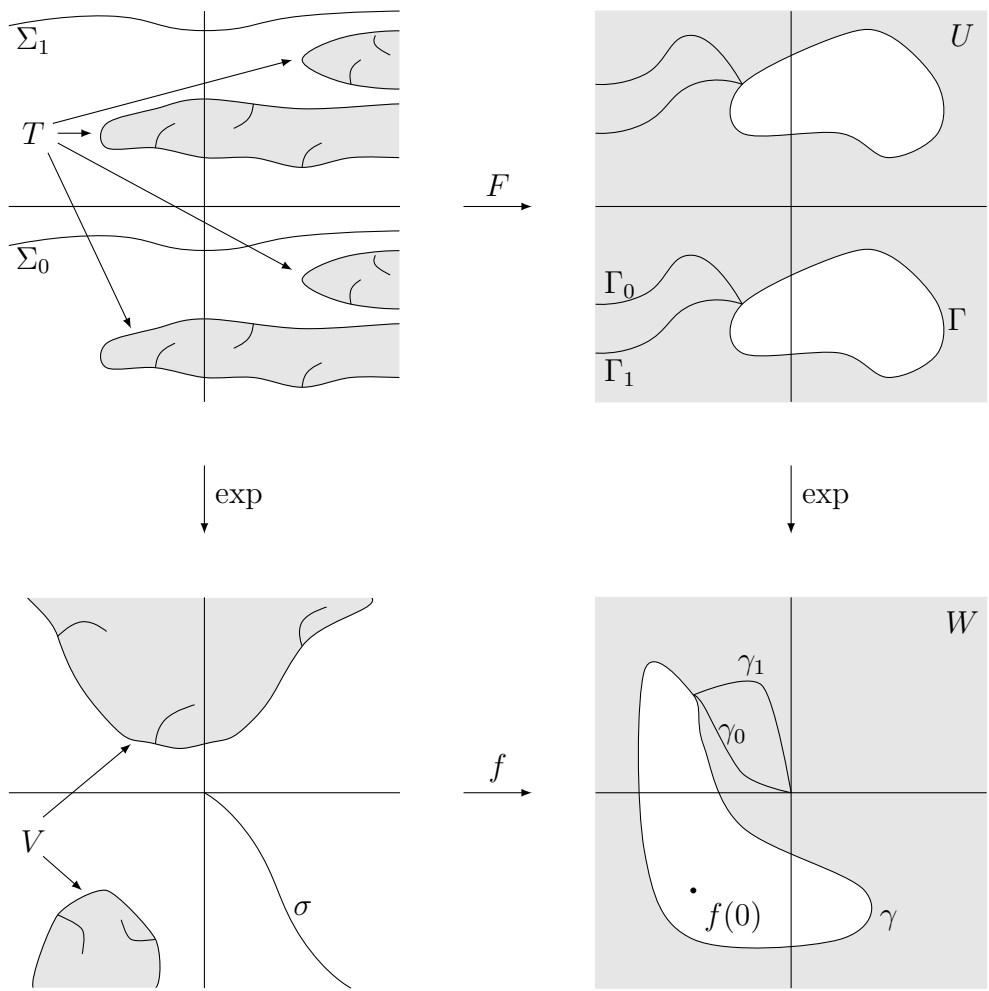
\begin{figure}[!htb]
\captionsetup{width=.85\textwidth}
\centering
\begin{tikzpicture}[scale=0.65,>=latex]
\filldraw [gray!20] (2,-10) rectangle (10,-2);
\filldraw [white] plot  [smooth cycle,tension=0.6] coordinates {(5,-6.7) (4.2,-5) (4,-4) (3,-3.5) (3,-7.2) (4,-9) (7,-9) (7.5,-8) };
\draw [black] plot  [smooth cycle,tension=0.6] coordinates {(5,-6.7) (4.2,-5) (4,-4) (3,-3.5) (3,-7.2) (4,-9) (7,-9) (7.5,-8) };
\draw [black] plot  [smooth, tension=0.5] coordinates {(6,-6) (5.4,-3.8) (4,-4) };
\draw [black] plot  [smooth, tension=0.5] coordinates {(6,-6) (5,-5.6) (4.2,-4.2) (4,-4) };
\filldraw [gray!20] (2,2) rectangle (10,10);
\filldraw [white] plot  [smooth cycle,tension=0.6] coordinates {(7,5) (8,5) (9,4) (9,3) (8,2.5) (7,3) (5,3) (5,4) };
\draw [black] plot  [smooth cycle,tension=0.6] coordinates {(7,5) (8,5) (9,4) (9,3) (8,2.5) (7,3) (5,3) (5,4) };
\filldraw [white] plot  [smooth cycle,tension=0.6] coordinates {(7,9.5) (8,9.5) (9,8.5) (9,7.5) (8,7) (7,7.5) (5,7.5) (5,8.5) };
\draw [black] plot  [smooth cycle,tension=0.6] coordinates {(7,9.5) (8,9.5) (9,8.5) (9,7.5) (8,7) (7,7.5) (5,7.5) (5,8.5) };
\draw [black] plot  [smooth, tension=0.9] coordinates {(2,7.5) (3,7.7) (4,8.5) (5,8.5) };
\draw [black] plot  [smooth, tension=0.9] coordinates {(2,3.0) (3,3.2) (4,4.0) (5,4.0) };
\draw [black] plot  [smooth, tension=0.9] coordinates {(2,4.0) (3,4.2) (4,5.0) (5,4.0) };
\draw [black] plot  [smooth, tension=0.9] coordinates {(2,8.5) (3,8.7) (4,9.5) (5,8.5) };
\filldraw [gray!20] plot  [smooth cycle,tension=0.9] coordinates {(-3,-1.8) (-4,-3) (-5,-4.5) (-6,-5) (-7,-5) (-8,-4.5) (-9,-2.8) (-9,-1.5) };
\draw [black] plot  [smooth cycle,tension=0.9] coordinates {(-3,-1.8) (-4,-3) (-5,-4.5) (-6,-5) (-7,-5) (-8,-4.5) (-9,-2.8) (-9,-1.5) };
\draw [black] plot  [smooth,tension=0.9] coordinates {(-4,-3) (-4,-2.5) (-3.5,-2.2)  };
\draw [black] plot  [smooth,tension=0.9] coordinates {(-7,-5) (-6.8,-4.5) (-6.2,-4.2)  };
\draw [black] plot  [smooth,tension=0.9] coordinates {(-9,-2.8) (-8.5,-2.5) (-8,-2.7)  };
\filldraw [white] (-10,-2) rectangle (-2,-1);
\filldraw [gray!20] plot  [smooth cycle,tension=0.4] coordinates {(-9.5,-10.0) (-9,-8.5) (-8,-8) (-7,-9) (-7,-10.0) (-8.2,-10.1) };
\draw [black] plot  [smooth cycle,tension=0.4] coordinates {(-9.5,-10.0) (-9,-8.5) (-8,-8) (-7,-9) (-7,-10.0) (-8.2,-10.1) };
\draw [black] plot  [smooth,tension=0.6] coordinates {(-9,-8.5) (-8.5,-8.8) (-8.5,-9.3)  };
\draw [black] plot  [smooth,tension=0.6] coordinates {(-7,-9) (-7.5,-9.0) (-7.7,-8.6)  };
\filldraw [white] (-10,-10.2) rectangle (-1,-10);
\draw [black] plot  [smooth,tension=0.9] coordinates {(-6,-6) (-5,-7) (-4,-9) (-3,-10)  };
\filldraw [gray!20] plot  [smooth cycle,tension=0.6] coordinates {(-8,7.7) (-8,7.2) (-7,7.2) (-6,7) (-5,7) (-4,6.8) (-3,7) (-2,7) (-1.8,7) (-1.8,8) (-2,8.1) (-4,8) (-6,8.2) (-7,8) };
\draw [black] plot  [smooth cycle,tension=0.6] coordinates {(-8,7.7) (-8,7.2) (-7,7.2) (-6,7) (-5,7) (-4,6.8) (-3,7) (-2,7) (-1.8,7) (-1.8,8) (-2,8.1) (-4,8) (-6,8.2) (-7,8) };
\draw [black] plot  [smooth,tension=0.9] coordinates {(-7,7.2) (-6.9,7.5) (-6.6,7.7)  };
\draw [black] plot  [smooth,tension=0.9] coordinates {(-4,6.8) (-3.9,7.1) (-3.6,7.3)  };
\draw [black] plot  [smooth,tension=0.9] coordinates {(-5,8.1) (-5.1,7.8) (-5.4,7.6)  };
\filldraw [white] (-2,6) rectangle (-1,8.5);
\filldraw [gray!20] plot  [smooth cycle,tension=0.6] coordinates {(-8,3.2) (-8,2.7) (-7,2.7) (-6,2.5) (-5,2.5) (-4,2.3) (-3,2.5) (-2,2.5) (-1.8,2.5) (-1.8,3.5) (-2,3.6) (-4,3.5) (-6,3.7) (-7,3.5) };
\draw [black] plot  [smooth cycle,tension=0.6] coordinates {(-8,3.2) (-8,2.7) (-7,2.7) (-6,2.5) (-5,2.5) (-4,2.3) (-3,2.5) (-2,2.5) (-1.8,2.5) (-1.8,3.5) (-2,3.6) (-4,3.5) (-6,3.7) (-7,3.5) };
\draw [black] plot  [smooth,tension=0.9] coordinates {(-7,2.7) (-6.9,3.0) (-6.6,3.2)  };
\draw [black] plot  [smooth,tension=0.9] coordinates {(-4,2.3) (-3.9,2.6) (-3.6,2.8)  };
\draw [black] plot  [smooth,tension=0.9] coordinates {(-5,3.6) (-5.1,3.3) (-5.4,3.1)  };
\filldraw [gray!20] plot  [smooth cycle,tension=0.6] coordinates {(-4,9) (-3,9.5) (-1.5,9.5) (-1.5,8.5) (-3,8.5) };
\draw [black] plot  [smooth cycle,tension=0.6] coordinates {(-4,9) (-3,9.5) (-1.5,9.5) (-1.5,8.5) (-3,8.5) };
\draw [black] plot  [smooth,tension=0.9] coordinates {(-3,8.5) (-3,8.8) (-2.8,9) };
\draw [black] plot  [smooth,tension=0.9] coordinates {(-2.5,9.5) (-2.5,9.3) (-2.2,9.1) };
\filldraw [gray!20] plot  [smooth cycle,tension=0.6] coordinates {(-4,4.5) (-3,5) (-1.5,5) (-1.5,4) (-3,4) };
\draw [black] plot  [smooth cycle,tension=0.6] coordinates {(-4,4.5) (-3,5) (-1.5,5) (-1.5,4) (-3,4) };
\draw [black] plot  [smooth,tension=0.9] coordinates {(-3,4) (-3,4.3) (-2.8,4.5) };
\draw [black] plot  [smooth,tension=0.9] coordinates {(-2.5,5) (-2.5,4.8) (-2.2,4.6) };
\draw [black] plot  [smooth,tension=0.9] coordinates {(-10,5.2) (-8,5.4) (-6,5.1) (-4,5.4) (-2,5.5)  };
\draw [black] plot  [smooth,tension=0.9] coordinates {(-10,9.7) (-8,9.9) (-6,9.6) (-4,9.9) (-2,10)  };
\filldraw [white] (-2,1.5) rectangle (-1,10);
\draw[->](-0.7,6)->(0.7,6);
\node at (0,6)[above]{$F$};
\draw[->](-0.7,-6)->(0.7,-6);
\node at (0,-6)[above]{$f$};
\draw[-](-10,6)->(-2,6);
\draw[-](-10,-6)->(-2,-6);
\draw[-](2,6)->(10,6);
\draw[-](2,-6)->(10,-6);
\draw[-](6,2)->(6,10);
\draw[-](-6,2)->(-6,10);
\draw[-](-6,-2)->(-6,-10);
\draw[-](6,-2)->(6,-10);
\draw[-](-10,6)->(-2,6);
\draw[->](-6,0.7)->(-6,-0.7);
\node at (-6,0)[right]{$\exp$};
\draw[->](6,0.7)->(6,-0.7);
\node at (6,0)[right]{$\exp$};
\node at (8.0,-8.6) {$\gamma$};
\node at (4.8,-4.8) {$\gamma_0$};
\node at (5.2,-3.4) {$\gamma_1$};
\node at (9.4,3.7) {$\Gamma$};
\node at (2.5,4.4) {$\Gamma_0$};
\node at (2.5,2.6) {$\Gamma_1$};
\node at (9.5,9.5) {$U$};
\node at (9.5,-2.5) {$W$};
\node at (-9.5,7.5) {$T$};
\draw[->](-9.1,7.7)->(-4.2,9);
\draw[->](-9,7.5)->(-8.3,7.5);
\draw[->](-9,7.3)->(-4.2,4.7);
\draw[->](-9.2,7.1)->(-7.5,3.5);
\node at (-9.5,4.9) {$\Sigma_0$};
\node at (-9.5,9.4) {$\Sigma_1$};
\node at (-9.5,-7) {$V$};
\draw[->](-9.3,-6.5)->(-7.5,-5);
\draw[->](-9.3,-7.3)->(-8.5,-8);
\filldraw [black] (4,-8) circle (0.05);
\node at (4,-8)[below right]{$f(0)$};
\node at (-4.4,-8)[below right]{$\sigma$};
\end{tikzpicture}
\caption{The logarithmic change of variable.}
\label{logchange}
\end{figure}

As mentioned, we call $F$ the logarithmic transform of $f$.
By construction, there exists $s>0$ such that $\H_{>s}\subset U$.
(In the introduction we sketched the case $\H_{>s}=U$.)
For $w\in T$ satisfying $\re F(w)>s$, let $G$ be the branch of the inverse of $F$ mapping
$F(w)$ to $w$.
Then $G$ is univalent in the disk $D(F(w),\re F(w)-s)$.
The image of this disk does not contain a disk of radius greater than $\pi$ 
around~$w$. 
Koebe's one quarter theorem \eqref{koebe3} now yields,
with $u=F(w)$, that 
\begin{equation} \label{l3a}
|G'(u)|\leq \frac{4\pi}{\re u-s}
\quad\text{for}\ \re u>s.
\end{equation}
In terms of $F$ we obtain~\eqref{e2}; that is,
\begin{equation} \label{l3b}
|F'(w)|\geq \frac{1}{4\pi}(\re F(w)-s)
\quad\text{if}\ w\in T\ \text{and}\ \re F(w)>s.
\end{equation}

The components of $V$ are called \emph{tracts} of $f$ and those of $T$ are called \emph{logarithmic tracts}.
Since the order of $f$ is finite,
the Denjoy--Carleman--Ahlfors theorem (see, e.g., \cite[Chapter~5, Section~1]{Goldberg2008}) 
implies that $V$ has only finitely many components; that is, $f$ has only finitely many tracts.
In fact, the number of tracts is at most $\max\{2\rho(f),1\}$.
We will not need this bound, however.

There exists a curve $\sigma$ in $\C\setminus V$ which connects $0$ with $\infty$.
The curve $\sigma$ may be wiggly. We note, however, that 
Ahlfors' spiral theorem~\cite[Section~8.5.1]{Hayman1989} gives an upper bound on how much the
components of $V$ and hence the
curve $\sigma$ may wind around the origin. We will not require the sharp bound provided
by this theorem. 
The results of \S\ref{geom-tract} below 
yield a (non-sharp) bound for this winding.
This bound suffices for our purposes.

Let $\Sigma$ be a preimage of $\sigma$ under the exponential function.
Then $\Sigma$ is a curve tending to $\infty$ on both ends,
with $\re w\to\infty$ as $w\to\infty$ through one end of $\Sigma$ 
while $\re w\to-\infty$ as $w\to\infty$ through the other end of $\Sigma$.
For $k\in\Z$ we put $\Sigma_k:=\Sigma+2\pi i k$.
The curves $\Sigma_k$ and $\Sigma_{k+1}$ then bound a ``strip-like'' domain~$S(k)$.

For $j\in\{0,\dots,p-1\}$ there exists $v_j\in S(0)\cap T$ such that 
$\exp v_j=z_j$. With $v_p:=v_0$ we then have 
\begin{equation} \label{l4}
\exp F(v_j)=f(\exp v_j)=f(z_j)=z_{j+1}=\exp v_{j+1}
\end{equation}
and thus 
\begin{equation} \label{l5}
F(v_j)=v_{j+1} + 2\pi m_{j+1}
\end{equation}
for some $m_{j+1}\in\Z$.
We put 
\begin{equation} \label{l6}
w_{j+1}:=v_{j+1} + 2\pi m_{j+1}
\end{equation}
for $j\in\{0,\dots,p-1\}$ and $w_0:=w_p$.
Then $w_j\in S(m_j)$ and $F(w_j)=w_{j+1}$.
Thus $\{w_0,\dots,w_{p-1}\}$ is a periodic cycle for $F$.
It follows from the equation $\exp\circ F^p=f^p\circ \exp$ that 
$(F^p)'(w_0)=(f^p)'(z_0)$ so that this periodic cycle is in fact 
attracting for~$F$.

Before defining the attracting basins of the $w_j$ we note that 
the endpoint $0$ of $\gamma_0$ is in $A^*(z_0,f)$ while the other endpoint
of $\gamma_0$ is on $\gamma$ and thus in $A^*(z_1,f)$.
Thus $\gamma_0$ intersects $J(f)$,
and so does $f^{-1}(\gamma_0)$.

For a reasonable definition of the Julia set $J(F)$ of $F$ we would expect that
\begin{equation} \label{l3c}
J(F)=\exp^{-1} J(f).
\end{equation}
We will in fact take~\eqref{l3c} as the definition of $J(F)$.
We find that $J(F)$ intersects both $\Gamma_0=\exp^{-1}(\gamma_0)$ 
and $\exp^{-1}(f^{-1}(\gamma_0))$.
Moreover, since $\partial A^*(z_0,f)$ and $\partial A^*(z_1,f)$ intersect $\gamma_0$, we
find that the boundaries of the (suitably defined) attracting basins of $w_0$ and $w_1$ will 
intersect $\Gamma_0$.

We thus want to extend the definition of $F$ to $\exp^{-1}(f^{-1}(\gamma_0))$.
In other words, we want to remove the ``spikes'' $\exp^{-1}(f^{-1}(\gamma_0))$ from~$T$.
Hence we put, similarly to~\eqref{l3}, 
\begin{equation} \label{l7}
W^*:=\exterior(\gamma),\  
V^*:=f^{-1}(W^*),\  
U^*:=\exp^{-1}(W^*)\  
\text{and} \ 
T^*:=\exp^{-1}(V^*).
\end{equation}

The function $F$ does not have a holomorphic (and not even a continuous) extension to~$T^*$.
In fact, the ``endpoints'' of the spikes are given by the set $X:=\exp^{-1}(f^{-1}(0))$.
And for $\xi\in X$ we have $\re F(w)=\log |f(e^w)|\to -\infty$ as $w\to\xi$.
Moreover, a continuous extension does not exist to points on the spikes which are not
endpoints.
In fact, let $\tau$ be a component of $\exp^{-1}(f^{-1}(\gamma_0\setminus \{0\}))$ 
and $\xi\in\tau$.
Then $F(w)$ tends to some point $\eta\in \exp^{-1}(\gamma_0)$ as $w\to\xi$ from one side
of $\tau$, while $F(w)$ tends to $\eta+2\pi i$ or $\eta-2\pi i$ as  $w\to\xi$ from 
the other side of $\tau$.

We will, however, consider a discontinuous extension of $F$ by defining 
$F(\xi)$ as one of these limiting values $\eta$ or $\eta\pm 2\pi i$.
To be definite, we choose the one with larger imaginary part.

We have thus extended $F\colon T\to U$ to a map $F\colon T^*\setminus X\to U^*$. 
The map $F$ is discontinuous on the spikes that form $T^*\setminus T$.
And it can be continued analytically across each such spike from both sides. 
In order to be able to work with $F$ and its iterates as a holomorphic function
also on the spikes (and on points mapped to the spikes), we introduce a
modification $F_1$ of $F$ obtained by perturbing the curve~$\gamma$.

To be precise, we consider a second 
curve $\gamma_1$ which connects the same point on $\gamma$ to $0$ as $\gamma_0$ does,
which is disjoint from $\gamma_0$ except for the endpoints, and which is homotopic
(with fixed endpoints) to $\gamma_0$ 
in $\C\setminus \{z_0,\dots,z_{p-1}\}$. (As mentioned, we can think of $\gamma_1$ as a small
perturbation of $\gamma_0$.)
The curves $\gamma_0$ and $\gamma_1$ then form a Jordan curve which bounds 
a domain $G$.
See again Figure~\ref{logchange}, where $\Gamma_1:=\exp^{-1}(\gamma_1)$.

We define $W_1$, $V_1$, $U_1$ and $T_1$ as in~\eqref{l3}, with $\gamma_0$ replaced
by $\gamma_1$,
and obtain a map $F_1\colon T_1\to U_1$.
We can extend $F_1$  to a map $F_1\colon T^*\setminus X\to U^*$ as above.
We can choose $F_1$ in such a way that if $w\notin \exp^{-1}(f^{-1}(\overline{G}))$,
then $F_1(w)=F(w)$.
Note that the condition $w\notin \exp^{-1}(f^{-1}(\overline{G}))$ is equivalent to
$F(w)\notin \exp^{-1}(\overline{G})$. 
On the other hand, if $w\in \exp^{-1}(f^{-1}(\overline{G}))$,
then $F_1(w)$ and $F(w)$ differ by $2\pi i$, yielding in particular that
$F_1(w)\in T^*\setminus X$ if and only if $F(w)\in T^*\setminus X$. 
In any case we find that $F(F_1(w))=F(F(w))$.

Suppose now that $w,F(w),F^2(w),\dots,F^{n-1}(w)\in T^*\setminus X$ so that the iterate
$F^n(w)\in U^*$ is defined.
The above reasoning shows that if $g_k\in\{F,F_1\}$ for $k\in\{1,\dots,n\}$,
then $F^n(w)=(F\circ g_{n-1}\circ \dots \circ g_1)(w)$ for all $w\in T^*\setminus X$.
If $F^{n-1}(w)\notin \exp^{-1}(f^{-1}(\overline{G}))$ or, equivalently,
$F^n(w)\notin \exp^{-1}(\overline{G})$, then we also have
$F^n(w)=(g_n\circ g_{n-1}\circ \dots \circ g_1)(w)$.
By a suitable choice of the $g_k$ we can now achieve that 
$g_n\circ g_{n-1}\circ \dots \circ g_1$ and thus
$F^n$ maps a neighborhood of $w$ biholomorphically onto a neighborhood of $F^n(w)$.

We put 
\begin{equation} \label{l12}
I'(F) := \{w\in T\colon  \re F^n(w) \to \infty \text{ as } n \to \infty \}
\end{equation}
and, for $M>0$,
\begin{equation} \label{l13}
I'(F,M):=\left\{w \in T\colon \liminf_{n\to \infty}\re F^n(w)\geq M\right\}.
\end{equation}
Analogously to~\eqref{b0c} we have 
\begin{equation} \label{l14}
I'(F)=\bigcap_{M>0}I'(F,M).
\end{equation}
We note that 
\begin{equation} \label{l15}
I(f)= \exp I'(F) 
\quad\text{and}\quad
I(f,M)= \exp I'(F,\log M). 
\end{equation}

\section{Attracting basins and itineraries}\label{att-basin}
By construction, $\gamma$ and its interior are in $A(z_1,f)$. Thus
\begin{equation} \label{a2}
\C\setminus V^*\subset A(z_0,f).
\end{equation}
On the other hand, $z_0\in V^*$ since $z_1$ is in the exterior of~$\gamma$.
Thus $w_0\in T^*=\exp^{-1}(V^*)$. For $w$ close to $w_0$ all iterates $F^n(w)$ are 
defined and $\lim_{k\to\infty} F^{kp}(w)=w_0$.
In view of~\eqref{a2} it seems reasonable to define the attracting basin $A(w_0,F)$ of $w_0$ with
respect to $F$ not only as the set of points with the last property, but
in such a way that it contains  $\C\setminus T^*= \exp^{-1}(\C\setminus V^*)$.
We thus put 
\begin{equation} \label{a3}
\begin{aligned} 
A(w_0,F) := \!\big\{ w\in\C\colon &
 F^{kp}(w)\in \C\setminus T^* \ \text{for some}\ k\in\N_0
\\ &
\text{or}\  
\lim_{k\to\infty} F^{kp}(w)=w_0
\big\}
\end{aligned} 
\end{equation}
and
\begin{equation} \label{a4}
A(w_j,F) := \left\{ w\in\C\colon 
F^{p-j}(w)\in A(w_0,F) \right\}
\end{equation}
for $1\leq j\leq p-1$.
\begin{lemma}\label{la-AF}
$A(w_j,F) =\exp^{-1}(A(z_j,f))$ for $j\in\{0,\dots,p-1\}$.
\end{lemma}
\begin{proof}
First we consider the case $j=0$.
Let $w\in A(w_0,F)$ and put $z:=e^w$.
If there exists $k\in\N_0$ such that $F^{kp}(w)\in \C\setminus T^*$, then
\begin{equation} \label{a6}
f^{kp}(z)=\exp F^{kp}(w) \in \exp(\C\setminus T^*) \subset \C\setminus V^*\subset A(z_0,f) 
\end{equation}
and hence $z\in A(z_0,f)$.
If there is no $k\in\N_0$ such that $F^{kp}(w)\in \C\setminus T^*$
and thus $F^{kp}(z)\in T^*$ for all $k\in\N_0$, then
$F^{kp}(w)\to w_0$ as $k\to\infty$ and thus 
\begin{equation} \label{a7}
f^{kp}(z)=\exp F^{kp}(w) \to \exp w_0 =z_0
\end{equation}
as $k\to\infty$. Thus we have $z\in A(z_0,f)$ in both cases.
Hence 
$A(w_0,F) \subset \exp^{-1}(A(z_0,f))$. 

To prove the opposite inclusion, suppose that $z=e^w\in A(z_0,f)$.
If there exists $k\in\N_0$ such that $\exp F^{kp}(w)=f^{kp}(z)\in \C\setminus V^*$, then
$F^{kp}(w)\in \exp^{-1}(\C\setminus V^*)=\C\setminus T^*$ and thus $w\in A(w_0,F)$.
If there does exist $k\in\N_0$ such that $\exp F^{kp}(w)=f^{kp}(z)\in \C\setminus V^*$, then
$F^{kp}(w)\in T^*$ for all $k\in\N_0$.
Since $f^{kp}(z)\to z_0$ we find that $F^{kp}(w)\to w_0$.
Thus $w\in A(w_0,F)$ in both cases.
Hence 
$\exp^{-1}(A(z_0,f)) \subset A(w_0,F)$.
This completes the proof for the case that $j=0$. 

The case that $1\leq j\leq p-1$ follows from this. In fact, for $z=e^w$ we find that if
$F^{p-j}(w)\in A(w_0,F)$,
then  $f^{p-j}(z)=\exp F^{p-j}(w)\in \exp A(w_0,F)\subset A(z_0,f)$
so that  $z\in A(z_j,f)$. Thus
$A(w_0,F) \subset \exp^{-1}(A(z_0,f))$. 
The opposite inclusion can be proved analogously.
\end{proof}

A standard result of complex dynamics (see \cite[Corollary~4.12]{Milnor2006}) says that 
\begin{equation} \label{a9}
J(f)=\partial A(z_j,f)
\end{equation}
for all $j\in\{0,1,\dots,p-1\}$.
Together with~\eqref{l3c} and Lemma~\ref{la-AF} this yields that 
\begin{equation} \label{a10}
J(F)=\partial A(z_j,F)
\end{equation}
for all $j\in\{0,1,\dots,p-1\}$.

Let now $w\in J(F)$. The \emph{itinerary} of $w$ is the sequence $(i_k)$ defined 
by $F^k(w)\in S(i_k)$.
We also write $i_k(w)$ instead of $i_k$.

Since $\sigma\in \C\setminus V^*$ we have $\Sigma_k\subset \C\setminus T^*\subset A(w_0,F)$ for all $k\in\Z$.
This implies that if $j\in\{1,\dots,p-1\}$ and if $A$ is a component of $A(w_j,F)$, 
then there exists $m\in\Z$ such that
$A\subset S(m)$ and in fact $\overline{A}\subset S(m)$.

Let $A_j:=A^*(w_j,F)$ be the component of $A(w_j,F)$ that contains $w_j$.
Since $w_j\in S(m_j)$
we have $\partial A_j\subset \overline{A_j}\subset S(m_j)$ for $1\leq j\leq p-1$.
We deduce that if $w\in\partial A_0$ and $1\leq j\leq p-1$,
then the itinerary $(i_k(w))$ satisfies $i_k(w)=m_j$ if $k\equiv j \pmod{p}$.
Thus if $w\in\partial A_1$,
then $i_k(w)=m_j$ if $k+1\equiv j \pmod{p}$ and $j\in\{1,\dots,p-1\}$.
In particular, 
\begin{equation} \label{a5}
F^{kp}(w) \in\partial A_1 \subset S(m_1)
\quad\text{for}\ w\in\partial A_1 \ \text{and}\ k\in\N.
\end{equation}

\section{Geometry of the tracts}\label{geom-tract}
We use arguments from papers by Bara\'nski~\cite{Baranski2007} and
Rottenfusser,  R\"uckert,  Rempe and Schleicher~\cite{Rottenfusser2011}
to show that the hypothesis
that $f$ has finite order implies that the tracts are not too wild.
\begin{lemma}\label{la-geom}
Let $f\in\EL$ be of finite order.
Then there exists $C_0>0$ such that every logarithmic tract $L$ of $f$
contains a curve $\tau$ tending to $\infty$ such that if $r>0$, then the length
of the part of $\tau$ which is contained in $\H_{\leq r}$ is at most~$C_0r$.
\end{lemma}
\begin{proof}
By hypothesis, \eqref{e4}  holds for
 $\mu>\rho(f)$ if $|z|$ is sufficiently large.
In terms of the logarithmic transform $F$ this means that 
\begin{equation} \label{g2}
\re F(z) \leq \exp (\mu\re z)
\end{equation}
if $\re z$ is sufficiently large, say $\re z>M$.
We may assume that $M$ is chosen such that $M>0$ and 
$\H_{>M}\subset U$.
Let $G\colon U\to L$ be the inverse of $F\colon L\to U$.
We define the curve $\tau$ by $\tau\colon [0,\infty)\to L$, $\tau(t)=G(M+1+t)$.
If $\re \tau(t)\leq r$, then
\begin{equation} \label{g4}
M+1+t=F(\tau(t))=\re F(\tau(t))\leq \exp(\mu\re\tau(t))\leq \exp(\mu r)
\end{equation}
and hence 
\begin{equation} \label{g5}
t\leq t(r):=\exp(\mu r)-1.
\end{equation}
Thus the part of $\tau$ which is contained in  $\H_{\leq r}$
is contained in the subcurve $\tau|_{[0,t(r)]}$.

Using \eqref{c5d} and \eqref{c5e} we find that
\begin{equation} \label{g6}
\begin{aligned} 
\ell_L(\tau|_{[0,t(r)]})
&=\ell_U([M+1,M+1+t(r)])
\\ &
\leq 
\ell_{\H_{>M}}([M+1,M+1+t(r)])
\\ &
=\int_{M+1}^{M+1+t(r)} \frac{dt}{\dist(t,\partial \H_{>M})}
=\int_{1}^{1+t(r)} \frac{dt}{t}
\\ &
=\log(1+t(r))
=\mu r.
\end{aligned} 
\end{equation}
On the other hand,
by \eqref{c5b} and since $\dist(z,\partial L)\leq\pi$ for all $z\in L$, we have
\begin{equation} \label{g7}
\begin{aligned} 
\ell_L(\tau|_{[0,t(r)]})
&=\int_{\tau|_{[0,t(r)]})} \lambda_L(z)|dz|
\\ &
\geq \frac{1}{2\pi} \int_{\tau|_{[0,t(r)]})} |dz|
=\frac{1}{2\pi} \ell(\tau|_{[0,t(r)]})) .
\end{aligned} 
\end{equation}
Combining \eqref{g6} and \eqref{g7} yields the conclusion with $C_0=2\pi\mu$.
\end{proof}

\section{Quasiconformal modification}\label{quasi}
As before we assume that the hypothesis  of Theorem~\ref{thm1} hold for $j=1$; that is,
$\sing(f^{-1})\subset A^*(z_1,f)$.
We mentioned at the beginning of \S\ref{log-change} 
that we may assume without loss of generality that $z_1\notin\sing(f^{-1})$.
In this section we will explain why this can be assumed.
So suppose we have $z_1\in\sing(f^{-1})$.
The idea is to modify the function $f$ quasiconformally so that the condition 
is satisfied for the modified function. There are several ways to do this 
quasiconformal modification. We shall use the following result
of Rempe and Stallard~\cite[Corollary~2.2]{Rempe2010}, which is proved 
using the results of Rempe's paper~\cite{Rempe2009}. 
Here two entire functions $f$ and $g$ are called \emph{affinely equivalent}
if there exist affine functions $\varphi,\psi\colon \C\to\C$ such that
$\psi\circ f=g\circ\varphi$.
Moreover,
\begin{equation} \label{rs2}
J_R(f):=\{z\in J(f)\colon |f^n(z)|\geq R\ \text{for all}\ n\geq 1\}.
\end{equation}

\begin{lemma}\label{la-qu1}
Suppose that $f,g\in\EL$ are affinely equivalent and let $K>1$. 
Then there exist $R>0$ and a $K$-quasiconformal map
$\theta\colon\C\to\C$ such that
\begin{equation} \label{rs1}
\theta(f(z)) = g(\theta(z))
\end{equation}
for all $z \in J_R(f)$.
\end{lemma}
We apply this result for $g_c(z):=f(z+c)$ with $c\in\C$.
Then $\sing(g_c^{-1})=\sing(f^{-1})$.
For sufficiently small $c$ the function $g_c$ will have an attracting
periodic point $z_{1,c}$ of period $p$ close to $z_1$
such that $A^*(z_{1,c},g_c)$ contains $\sing(g_c^{-1})$.
Since $z\mapsto z_{1,c}$ is not constant, we have $z_{1,c}\notin \sing(g_c^{-1})$ for small~$c$.

Rempe and Stallard~\cite[Theorem~1.3]{Rempe2010} used their Lemma~\ref{la-qu1}
to prove that if $f,g\in\EL$ are affinely equivalent, then $\dim I(f)=\dim I(g)$.
The same argument yields that 
\begin{equation} \label{rs3}
\lim_{M\to\infty} 
\dim (\partial A^*(z_{1,c},g_c)\cap I(g_c,M)) = 
\lim_{M\to\infty} 
\dim (\partial A^*(z_1,f)\cap I(f,M)) .
\end{equation}
We may thus suppose that $z_1\notin\sing(f^{-1})$ in the proof of Theorem~\ref{thm1}
since otherwise we can pass from $f$ to $g_c$ for some small~$c$.

\section{Proof of Theorem \ref{thm1}: Upper bound}\label{hausdimI}
In this section we will prove that, under the hypotheses of Theorem~\ref{thm1}, we have
\begin{equation} \label{b4<}
\lim_{M\to\infty} 
\dim (\partial A^*(z_j,f)\cap I(f,M)) \leq 1 .
\end{equation}
We begin with the following lemma.
\begin{lemma}\label{la-cov}
For each $\alpha >1$ there exists $M>0$ with the following property:
If $a\in\H_{>M}$,
if $0<r\leq 8$ and if $D(a,r)\cap A_1\neq \emptyset$, 
then there exist disks $D(a_k,r_k)$, $k\in\Z$, such that 
\begin{equation} \label{u1}
F^{-p}(D(a,r))\cap\partial A_1 \subset \bigcup_{k\in\Z} D(a_k,r_k)
\end{equation}
and
\begin{equation} \label{u2}
\sum_{k\in\Z} r_k^\alpha \leq \frac{r^\alpha}{M^\alpha} .
\end{equation}
\end{lemma}
\begin{proof}
Choose $R>0$ so large that $\H_{>R}\subset U$.
Then the branches of $F^{-1}$ are holomorphic in $\H_{>R}$.
For large $M>R$ such a branch of $F^{-1}$ maps $\H_{>M}$ into $\H_{>R}$.
Thus the branches of $F^{-2}=(F^{-1})^2$ are holomorphic in $\H_{>M}$.
Inductively we see that for sufficiently large $M$
the branches of $F^{-p}$ are holomorphic in $\H_{>M}$.

Since there are only finitely many, say $m$, logarithmic tracts in each domain $S(l)$
the set of preimages of $a$ under $F$ may be written in the form 
$\{b_j+2\pi i k\colon  k\in\Z, 1\leq j\leq m\}$.
We may assume here that $|\im b_j|\leq\pi$ for $1\leq j\leq m$.
Let $G_j$ be the branch of $F^{-1}$ that maps $a$ to $b_j$.
It follows from~\eqref{l3a} that if $u\in D(a,r)$, then 
\begin{equation} \label{u3}
|G_j'(u)|\leq \frac{4\pi}{\re u -R}\leq \frac{4\pi}{M-8-R}\leq \frac{13}{M}, 
\end{equation}
provided $M$ is sufficiently large.
Hence
\begin{equation} \label{u4}
G_j(D(a,r))\subset D\!\left(b_j,\frac{13}{M}r\right).
\end{equation}
This implies that 
\begin{equation} \label{u5}
F^{-1}(D(a,r))\subset \bigcup_{j=1}^m \bigcup_{k\in\Z} D\!\left(b_j+2\pi i k,\frac{13}{M}r\right).
\end{equation}
Choosing $M$ large we have $D(b_j+2\pi i k,13r/M)\subset \H_{>R}$ for all $j$ and $k$.
In fact, 
\begin{equation} \label{u5a}
\re b_j\geq \frac{\log M}{\mu}
\end{equation}
 by~\eqref{g2}.
Let $G_*$ be the branch of $F^{-1}$ which maps $\H_{>R}$ into the logarithmic tract 
containing $A_{p-1}$. 
It follows from~\eqref{b2} that if $u\in D(b_j+2\pi i k,13r/M)$, then
\begin{equation} \label{u6}
|G_*'(u)|=\frac{1}{|F'(G_*(u))|}\leq \frac{1}{\beta|F(G_*(u))|}
= \frac{1}{\beta|u|}.
\end{equation}
Noting that  $|z|\geq (\re z+|\im z|)/2$ and using~\eqref{u5a} as well as $|\im b_j|\leq \pi$
we see that 
\begin{equation} \label{u7}
\begin{aligned}
|u|
&\geq |b_j+2\pi i k|-\frac{13}{M}r
\\ &
\geq \frac12\!\left( \frac{\log M}{\mu}+2\pi|k|-\pi\right)-\frac{13}{M}r 
\geq \frac13\!\left( \frac{\log M}{\mu}+\pi|k|\right)
\end{aligned}
\end{equation}
for $u\in D(b_j+2\pi i k,13r/M)$, provided $M$ is sufficiently large.
Combining the last two estimates we find that 
\begin{equation} \label{u8}
|G_*'(u)|\leq
\frac{3\mu}{\beta(\log M +\mu\pi|k|)}
\end{equation}
for $u\in D(b_j+2\pi i k,13r/M)$.
Putting $c_{jk}:=G_*(b_j+2\pi i k)$ and 
\begin{equation} \label{u9}
s_k:=\frac{39\mu}{\beta(\log M +\mu\pi|k|)M}r
\end{equation}
we conclude that 
\begin{equation} \label{u10}
F^{-2}(D(a,r))\subset \bigcup_{j=1}^m \bigcup_{k\in\Z} D\!\left(c_{jk}, s_k\right).
\end{equation}
Finally, let $G_{**}$ be the branch of $(F^{p-2})^{-1}$ that maps $A_{p-1}$ 
to $A_1$. (If $p=2$, then $G_{**}$ is the identity.)
By~\eqref{b2}, or the weaker standard estimate~\eqref{e2}, we have
$|F'(w)|\geq 1$ if $\re F(w)$ is large enough.
Thus $|(F^{-1})'(u)|\leq 1$ if $\re u$ is large enough.
We conclude that $|G_{**}'(u)|\leq 1$ if $\re u$ is large enough.
In particular, choosing $M$ sufficiently large
we can achieve that $|G_{**}'(u)|\leq 1$ for all $u$ in one of the 
disks $D\!\left(c_{jk}, s_k\right)$.
With $d_{jk}:=G_{**}(c_{jk})$ we thus have $G_{**}(D(c_{jk}, s_k))\subset D(d_{jk},s_k)$.
We deduce that 
\begin{equation} \label{u11}
F^{-p}(D(a,r))\cap\partial A_1 \subset \bigcup_{j=1}^m \bigcup_{k\in\Z}  D(d_{jk},s_k) .
\end{equation}
It follows easily from~\eqref{u9} that
\begin{equation} \label{u12}
m \sum_{k\in\Z}  s_k^\alpha \leq \frac{r^\alpha}{M^\alpha}
\end{equation}
for large~$M$.
The conclusion follows by rearranging the $d_{jk}$ into one sequence $(a_k)_{k\in\Z}$,
obtaining the $r_k$ from the $s_k$ accordingly.
\end{proof}
\begin{remark}\label{remark1}
We note that it was in~\eqref{u6} where the hypothesis~\eqref{b2} was used 
and where the standard estimate~\eqref{e2} for functions in $\EL$ is not sufficient.

As already mentioned, we do not know whether the hypothesis~\eqref{b2}
is necessary.
Another condition involving a lower bound for the derivative is the 
\emph{rapid growth} condition introduced by Mayer and Urba\'nski~\cite{Mayer2008}. 
It appears to us that our condition is of a different nature than theirs.
\end{remark}
Given a covering of $\H_{>M}\cap \partial A_1$ with disks,
Lemma~\ref{la-cov} yields a new covering of $\H_{>M}\cap\partial A_1$.
In order to use this to estimate the Hausdorff dimension, we need 
a covering where we can start with.
This is given by the following lemma.
\begin{lemma}\label{la-cov2}
For each $\alpha >1$ there exist $M>0$ and
disks $D(a_k,r_k)$, $k\in\Z$, such that
\begin{equation} \label{u13}
F^{-p}(\H_{>M} \cap \partial A_1)\cap\partial A_1
\subset \bigcup_{k\in\Z} D(a_k,r_k)
\end{equation}
and
\begin{equation} \label{u14}
\sum_{k\in\Z} r_k^\alpha <\infty .
\end{equation}
\end{lemma}
\begin{proof} 
Choose $M$ according to Lemma~\ref{la-cov}. Let $n\in\N$ with $2^{n-1}\geq M$.
Let $L$ be the logarithmic tract containing $A_1$ and let the curve $\tau$ and
the constant $C_0$ be as in Lemma~\ref{la-geom}.
This lemma yields that if $n$ is large enough, then the length of the intersection
of $\tau$ with $\{z\colon 2^{n-1}\leq\re z\leq 2^{n+2}\}$ 
is at most $C_02^{n+2}$. 
This intersection may consist of several pieces.
Each piece which intersects 
$P_n:=\{z\colon 2^{n}\leq\re z\leq 2^{n+1}\}$ 
must have length at least $2^{n-1}$. Thus there are at 
most $8C_0$ such pieces. 

Since such a piece has length at most $C_02^{n+2}$, we can cover its 
intersection with $P_n$ by at most $C_0 2^{n+2}+1$ disks of radius~$1$.
Altogether we can cover the intersection of $\tau$ with $P_n$ by 
$K_n$ disks of radius~$1$, where
\begin{equation} \label{u15}
K_n \leq 8C_0(C_02^{n+2}+1)\leq 33C_0^2 2^n ,
\end{equation}
if $M$ and hence $n$ are large.
We may assume that the centers of these disks are contained in $P_n$.
Since for each $w\in\partial A_1$ there exists $w'\in\tau$ with $\re w=\re w'$ and
$|w-w'|=|\im(w-w')|\leq 2\pi$ we deduce that we can cover $P_n\cap \partial A_1$ by 
$K_n$ disks of radius $2\pi+1$.

We apply Lemma~\ref{la-cov} to each disk $D(a,2\pi+1)$ used in this 
covering of $P_n\cap \partial A_1$. As there are $K_n$ such disks,
we obtain a covering of 
$F^{-p}(P_n\cap \partial A_1)\cap\partial A_1$  with disks $D(a_k,r_k)$
such that 
\begin{equation} \label{u16}
\sum_{k\in\Z} r_k^\alpha \leq K_n \frac{(2\pi+1)^\alpha}{(2^n)^\alpha}
\leq \frac{33C_0^2(2\pi+1)^\alpha}{2^{(\alpha-1)n}}.
\end{equation}
The conclusion follows by taking the sum over all $n$ with $2^{n-1}\geq M$.
\end{proof} 
\begin{proof}[Proof of the upper bound in Theorem~\ref{thm1}] 
In order to prove~\eqref{b4<} it suffices in
view of~\eqref{l15} and Lemma~\ref{la-AF} to show that 
\begin{equation} \label{u17}
\lim_{M\to\infty} \dim (\partial A_1\cap I'(F,M)) \leq 1 .
\end{equation}
Let 
\begin{equation} \label{u18}
I^*(F,M):=\left\{w \in T\colon \re F^k(w)\geq M \ \text{for all}\ k\in\N \right\}.
\end{equation}
With 
\begin{equation} \label{u18a}
I_n:=\left\{w \in T\colon \re F^k(w)\geq M \ \text{for}\ 1\leq k\leq n \right\}.
\end{equation}
we thus have 
\begin{equation} \label{u18b}
I^*(F,M)=\bigcap_{n\in\N} I_n .
\end{equation}

Since 
\begin{equation} \label{u19}
I'(F,M)\subset \bigcup_{n\in\N} F^{-n}(I^*(F,M-1))
\end{equation}
it suffices to prove that 
\begin{equation} \label{u20}
\lim_{M\to\infty} \dim (\partial A_1\cap I^*(F,M)) \leq 1 .
\end{equation}
In order to do so, let $\alpha>1$. Lemma~\ref{la-cov2} says that if $M$ is sufficiently 
large, then $I_1\cap \partial A_1$ can be covered by disks $D(a_k,r_k)$ such that the radii satisfy
\begin{equation} \label{u21}
S:=\sum_{k\in\Z} r_k^\alpha <\infty .
\end{equation}
Applying Lemma~\ref{la-cov} to each disk for which $\re a_k\geq M$ yields a covering
of $I_2\cap \partial A_1$ with disks $D(b_k,s_k)$ such that
\begin{equation} \label{u22}
\sum_{k\in\Z} s_k^\alpha <\frac{S}{M^\alpha}.
\end{equation}
Inductively, given $n\in\N$, we obtain a covering of $I_n\cap \partial A_1$  with disks $D(c_k,t_k)$
such that
\begin{equation} \label{u23}
\sum_{k\in\Z} t_k^\alpha <\frac{S}{M^{n\alpha}}.
\end{equation}
Since we may assume that $M>1$ this implies in particular that the radii of the disks tend to
$0$ as $n$ tends to~$\infty$.
It follows that $\dim I^*(F,M)\leq \alpha$. Thus~\eqref{u20} follows. 
\end{proof} 
\begin{remark} 
Many of the arguments used apply more generally to functions in the Eremenko--Lyubich class~$\EL$.
We have used the hypothesis that $f\in\Speiser$ only to ensure that 
$z_1\notin \sing(f^{-1})$ can be achieved by a small perturbation, and  that
if $\sing(f^{-1})\cup\{f(0)\}\subset A^*(z_1,f)\setminus\{z_1\}$,
then there exists a Jordan curve $\gamma$ in $A^*(z_1,f)$ which contains $\sing(f^{-1})$ and $f(0)$
in its interior,
but $z_1$ in its exterior.
\end{remark} 

\section{Proof of Theorem \ref{thm1}: Lower bound}\label{hausdimI2}
In order to prove that
\begin{equation} \label{b4>}
\lim_{M\to\infty} 
\dim (\partial A^*(z_j,f)\cap I(f,M)) \geq 1 
\end{equation}
it suffices to prove that $\partial A^*(z_1,f)\cap I(f)$ contains a continuum.
This follows if we show that $\partial A_1\cap I'(F)$ contains a continuum.

For $j\in\{1,\dots,p-1\}$, let $L_j$ be the component of $T^*$ that contains $A_j$. 
Let $L_0$ be any component of $T^*$.
For large $M$ we consider the set $X$ of all $w\in L_1\cap \H_{\geq M}$ such that 
$F^k(w)\in L_j\cap \H_{\geq M}$ if $k+1\equiv j \pmod{p}$.

\begin{lemma}\label{la-cont}
The set $X$ contains an unbounded closed connected subset of $I'(F)$.
\end{lemma}
This lemma can be deduced from a result given by Benini and Rempe~\cite[Theorem~2.5]{Benini2020}.
More precisely, it follows from part (c) of their theorem that $X\neq\emptyset$.
Once this is known, the 
existence of the unbounded closed connected subset of $I'(F)$ follows from part~(a).

Benini and Rempe note that the results are not entirely new and they give precise
references where 
these (or very similar) results can be found in the papers~\cite{Baranski2007a,Benini2015,Rempe2007,Rempe2008,Rippon2005}.
However, they also include a self-contained proof.

We mention that the results of Bara\'nski~\cite{Baranski2007} and
Rottenfusser,  R\"uckert,  Rempe and Schleicher~\cite{Rottenfusser2011} yield
that the set $X$ actually contains a curve which tends to infinity and
which is contained in $I'(F)$.

\begin{proof}[Proof of the lower bound in Theorem~\ref{thm1}] 
It suffices to show that the set $X$ defined above is contained in~$\partial A_1$.
So let $u\in X$. 
For $k\in\N$ let $G_k$ be the branch of the inverse of $F^{kp}$ that maps $F^{kp}(u)$
to~$u$. Since $F^{kp}(u)\in L_1$ and $A_1\subset L_1$, and since $L_1\subset S(m_1)$,
 there exist $w_k\in A_1$ and a vertical segment $\sigma_k$
of length less than $2\pi$ that connects $F^{kp}(u)$ and $w_k$.
It follows from~\eqref{l3a} that if $M$ is chosen large enough, then
$G_k$ is defined on $\sigma_k$ and 
\begin{equation} \label{u25}
\ell(G_k(\sigma_k))\leq \frac{1}{2^k}\ell(\sigma_k)\leq \frac{\pi}{2^{k-1}}.
\end{equation}
Put $v_k:=G_k(w_k)$. The choice of the branch $G_k$ yields that $v_k\in A_1$.
By~\eqref{u25} we have $|v_k-u|\to 0$ as $k\to\infty$. Assuming that $M>|w_1|$ we 
have $u\notin A_1$ by the definition of~$X$.
We conclude that $u\in\partial A_1$.
\end{proof}

\section{Functions satisfying the hypothesis of Theorem~\ref{thm1}}\label{hypo-thm1}
\begin{proposition} \label{prop1}
Let $f$ be of the form \eqref{b3} with polynomials $p$ and $q$.
Then there exist $t>0$ and $\beta>0$ such that \eqref{b2a} and hence~\eqref{b2} hold.
\end{proposition}
\begin{proof} 
Let $d$ be the degree of $q$ and let $c$ be the leading coefficient of $q$ so that
$q(z)\sim cz^d$ as $z\to\infty$.
For $1\leq k\leq d$ we put
\begin{equation} \label{v1}
\phi_k:=\frac{(2k+1)\pi -\arg c}{d}.
\end{equation}
It is well-known and 
easy to see that there exists $a_1,\dots,a_d\in\C$ such that for each $\varepsilon>0$ we have 
\begin{equation} \label{v2}
f(z)\to a_k
\quad \text{as}\ z\to\infty, \ |\arg z -\phi_k|\leq \frac{\pi}{2d}-\varepsilon.
\end{equation}
Moreover, the values $a_1,\dots,a_d$ are the only asymptotic values of $f$.
The only critical points of $f$ are the zeros of $p$. Thus $f\in\Speiser$.

An asymptotic expansion more precise than~\eqref{v2} is \cite[Lemma~4.1]{Hemke2005}
\begin{equation} \label{v3}
f(z)- a_k=(1+o(1))
\frac{p(z)}{q'(z)}\exp q(z)
\quad \text{as}\ z\to\infty, \ |\arg z -\phi_k|\leq \frac{\pi}{d}.
\end{equation}
The limits in~\eqref{v2} and~\eqref{v3} are uniform in the sectors specified.
If $R$ is sufficiently large and if $|f(z)|\geq  R$, we thus have
\begin{equation} \label{v4}
|f(z)|\leq 2\left| \frac{p(z)}{q'(z)}\exp q(z) \right| .
\end{equation}
Note here that if $R$ is large and $|f(z)|\geq R$, then $|z|$ is also large.
First~\eqref{v4} holds for those $z$ satisfying $|f(z)|\geq R$ which are in the sectors specified in~\eqref{v3},
but as these sectors cover the whole plane, we actually have~\eqref{v4}
for all $z\in\C$ for which $|f(z)|\geq R$.
Since $f'=pe^q$ we conclude that
\begin{equation} \label{v5}
\left|\frac{zf'(z)}{f(z)} \right|\geq \frac12 |zq'(z)| \quad\text{if} \ |f(z)|\geq R.
\end{equation}

For a branch of $\log f$ defined in a tract we find
for $z$ in the intersection of the tract with the sector 
given in~\eqref{v3} that
\begin{equation} \label{v6}
\log f(z) =  \log(f(z)-a_k) +O(1) = q(z) +\log\frac{p(z)}{q'(z)} +O(1)
\end{equation}
and hence that 
\begin{equation} \label{v7}
|\log f(z) | \leq  2|q(z)| .
\end{equation}
Together with~\eqref{v5} we thus have 
\begin{equation} \label{v8}
\left|\frac{zf'(z)}{f(z) \log f(z)} \right|\geq \frac14 \left| \frac{zq'(z)}{q(z)}\right|.
\end{equation}
We deduce that~\eqref{b2a} and hence~\eqref{b2} hold if $\beta<d/4$ and if $t$ is large enough.
\end{proof}

\section{Verification of Example \ref{ex1}} \label{veri}
A meromorphic function $f\colon\C\to\C\cup\{\infty\}$ is said to have 
a \emph{logarithmic singularity} over a point $a\in\C\cup\{\infty\}$ 
if there exists a neighborhood $U$ of $a$
and a component $V$ of $f^{-1}(U)$ such that $f\colon V\to U\setminus\{a\}$ is 
a universal covering. If this is the case for some neighborhood $U$ of~$a$, then it holds
for any simply connected neighborhood $U$ of $a$ which does not contain any other singularity
of $f^{-1}$.
We say that $V$ is a \emph{tract over the logarithmic singularity}~$a$.
(The term ``logarithmic tract'' is also common, but we have used this with
a different meaning in \S\ref{log-change}.)
There may be more than one logarithmic singularity over the same point, 
meaning that for some neighborhood $U$ there are several components $V$ of $f^{-1}(U)$
with the above property.

We note that if $f$ is an entire function in $\EL$, then all singularities 
over $\infty$ are logarithmic.

We summarize some results about the singularities of the inverse 
of the  functions $f_\lambda$ given by~\eqref{b7}.
\begin{lemma}\label{la-ex1}
The function $f_\lambda$ has two logarithmic singularities over $\infty$ 
and one logarithmic singularity over each of the points
$\lambda\sqrt{\pi}/2$ and $-\lambda\sqrt{\pi}/2$.
There are no other finite singularities of $f_\lambda^{-1}$; that is, 
$\sing(f_\lambda^{-1})=\{\pm\lambda\sqrt{\pi}/2\}$.

For a neighborhood $U$ of $\lambda\sqrt{\pi}/2$ contained in $\H_{>0}$
the corresponding tract $V$ satisfies $V\subset\H_{>0}$ if $\lambda>0$
and $V\subset\H_{<0}$ if $\lambda<0$

Moreover, $f_\lambda(z)\neq \lambda\sqrt{\pi}/2$ for $z\in \H_{>0}$.
\end{lemma}
\begin{proof}
It is well-known that $f_\lambda$ has logarithmic singularities as stated 
and that $\sing(f_\lambda^{-1})=\{\pm\lambda\sqrt{\pi}/2\}$;
see, e.g., \cite[p.~290]{Nevanlinna1953}.
The claim that the tract $V$ is contained in the half-plane specified if $U\subset \H_{>0}$ follows since 
$f_\lambda$ maps the imaginary axis into itself.

The last claim was proved by Hille~\cite[p.~34]{Hille1927}.
(It also follows from the description of the line complex given in~\cite[p.~368]{Nevanlinna1932}
or \cite[p.~306]{Nevanlinna1953}.)
\end{proof}

In order to study the case $\lambda=-2$ appearing in Example~\ref{ex1}, it will 
be convenient to consider the case $\lambda=2$ first. 
Since $f_2'(x)=2\exp(-x^2)>0$ and 
$f_2''(x)=-4x\exp(-x^2)<0$ for $x>0$, the function $f_2$ is 
increasing and concave on the interval $(0,\infty)$.
Since $f'(0)=2>1$ and $\lim_{x\to\infty}f_2(x)=\sqrt{\pi}$ this implies that 
$f_2$ has a unique fixed point $z_0\in(0,\infty)$ 
and that for this fixed point $z_0$ we have $0<f'(z_0)<1$.
Thus $z_0$ is an attracting fixed point of~$f_2$.
A numerical computation yields that $z_0\approx 1.7487\dots$.

Since $f_2$ is odd, $z_1:=-z_0$ is also an attracting fixed point of~$f_2$.
The attracting basin for $z_1$ is the reflection of that for $z_0$ at the imaginary axis.
Noting that $f_{-2}(z)=-f_2(z)$ we find that $f_{-2}(z_0)=z_1$ and $f_{-2}(z_1)=z_0$.
Hence $z_0$ and $z_1$ are periodic points of period $2$ of~$f_{-2}$.
Since $f_{-2}^2=f_2^2$, the periodic points $z_0$ and $z_1$ are also attracting
for $f_{-2}$, 
and the attracting basins and immediate attracting basins
of $z_0$ and $z_1$ with respect to $f_{-2}$ agree with those with respect to~$f_{2}$.
Thus in order to prove~\eqref{b8} it suffices to show that
\begin{equation} \label{ve1}
\dim(\partial A^*(z_0,f_{2})\cap I(f_{2}))=2.
\end{equation}
We note that $f_2$ has logarithmic singularities over $\pm\sqrt{\pi}$.

\begin{lemma}\label{la-ex2}
$f^{-1}(A^*(z_0,f_{2}))\cap \H_{>0}= A^*(z_0,f_{2})$.
\end{lemma}
\begin{proof}
Since $f_2$ leaves the imaginary axis invariant, $A(z_0,f_{2})$ does not intersect the 
imaginary axis so that $A^*(z_0,f_{2})\subset \H_{>0}$. 
As every periodic Fatou component of an entire function,
$A^*(z_0,f_{2})$ is simply connected.
Next we note that $\sqrt{\pi}\in A^*(z_0,f_2)$.
This is easily checked directly, but also follows from the result in complex dynamics
mentioned already in~\eqref{b6} which says
 that an immediate attracting basin contains a singularity of the inverse.
We conclude that one component of $f^{-1}(A^*(z_0,f_{2}))$ is a tract $V$ over the logarithmic 
singularity $\sqrt{\pi}$, as explained at the beginning of this section.
In fact, $V=A^*(z_0,f_{2})$.
All other components of $f^{-1}(A^*(z_0,f_{2}))$ contain a preimage of~$\sqrt{\pi}$.
Since all preimages of $\sqrt{\pi}$ are contained in $\H_{<0}$ by Lemma~\ref{la-ex1},
we find that $V=A^*(z_0,f_{2})$ is the only component of $f^{-1}(A^*(z_0,f_{2}))$ that intersects~$\H_{>0}$.
\end{proof}
The conclusion of Lemma~\ref{la-ex2} can also be expressed as
\begin{equation} \label{ve4}
A^*(z_0,f_{2})=\{z\in  A(z_0,f_{2})\colon f_2^n(z)\in \H_{>0}\ \text{for all}\ n\geq 0\}.
\end{equation}
Morosawa~\cite{Morosawa2004} showed that if a function of the form $f_\lambda(z)+\mu$, 
with $\lambda\in\R\setminus\{0\}$ and $\mu\in\R\cup i\R$, has two attracting fixed points or an 
attracting periodic point of period~$2$, then the boundaries of 
the immediate attracting basins contain a common curve.
It is easy to see that for symmetry reasons this common curve is the imaginary axis
in our example.
Thus we have the following result.
\begin{lemma}\label{la-ex3}
$\{iy\colon y\in\R\}\subset \partial A^*(z_0,f_{2})$.
\end{lemma}
Combining the last two lemmas we obtain the following result.
\begin{lemma}\label{la-ex4}
$\partial A^*(z_0,f_{2})=\{z\in \H_{\geq 0}\cap J(f_2)\colon f_2^n(z)\in \H_{\geq 0}\ \text{for all}\ n\geq 0\}$.
\end{lemma}
\begin{proof}
Let $X:=\{z\in \H_{\geq 0}\cap J(f_2)\colon f_2^n(z)\in \H_{\geq 0}\ \text{for all}\ n\geq 0\}$.
Lemma~\ref{la-ex2} yields that $\partial A^*(z_0,f_{2})\subset X$.

To prove the opposite inclusion, let $\xi\in X$.
In view of Lemma~\ref{la-ex3} we may assume that $\xi \in \H_{>0}$.
Let $U$ be a neighborhood of~$\xi$ contained in $\H_{>0}$.
Since $\xi\in J(f_2)$, there exists $n\in\N$ such that $f_2^n(U)\not\subset \H_{>0}$.
Let $n$ be minimal with this property.
Since $f_2^n(\xi)\in \H_{>0}$ we find that $f_2^n(U)$ intersects the imaginary axis.
Lemma~\ref{la-ex3} yields that there exists $z\in U$ with $f_2^n(z)\in A^*(z_0,f_{2})$.
It now follows from Lemma~\ref{la-ex2} that $z\in A^*(z_0,f_{2})$.
As $U$ can be taken as an arbitrarily small neighborhood of~$\xi$, we conclude 
that $\xi\in\partial A^*(z_0,f_{2})$.
\end{proof}
We now describe how~\eqref{ve1} can be deduced from Lemma~\ref{la-ex4}.
We already mentioned in the introduction that McMullen~\cite[Theorem~1.2]{McMullen1987}
 proved that $\dim J(\lambda e^z)=2$ for all $\lambda\in\C\setminus\{0\}$.
He actually showed that $\dim I(\lambda e^z)=2$ and used that
$I(\lambda e^z)\subset J(\lambda e^z)$.
Taniguchi~\cite[Theorem~1]{Taniguchi2003} extended this result to functions of the 
form~\eqref{b3}. Bara\'nski~\cite[Theorem~A]{Baranski2008} and
Schubert~\cite{Schubert2007} showed that it holds for functions in $\EL$
of finite order. Note that $I(f)\subset J(f)$ for $f\in\EL$, as was 
proved by Eremenko and Lyubich~\cite[Theorem~1]{Eremenko1992} using the logarithmic 
change of variable.
Thus we have the following result.
\begin{lemma}\label{la-ex5}
Let $f\in\EL$ be of finite order. Then $\dim I(f)=2$.
Moreover, for every tract $U$ of $f$ the set of all points $z\in U\cap I(f)$
for which $f^n(z)\in U$ for all $n\in\N$ has Hausdorff dimension~$2$.
\end{lemma}
In the papers cited the result is only stated with $J(f)$ instead of $I(f)$,
but the proofs also yield the above result.
For a statement (and proof) of the above result with $I(f)$ instead of $J(f)$ we also
refer to~\cite[Theorem~1.2]{Bergweiler2009}. There it is shown that 
the condition that $f$ has finite order can be weakened.

The function $f_2$ has two tracts, one containing $\{iy\colon y>y_0\}$ and one
containing $\{iy\colon y<-y_0\}$ for some $y_0>0$.
The tracts are symmetric with respect to the imaginary axis.
Lemma~\ref{la-ex4} says that in order to estimate  the Hausdorff dimension
of $\partial A^*(z_0,f_{2})$, we have to consider the set of escaping points which stay 
in a ``half-tract'' (or, more precisely, the set of points staying in two 
``half-tracts'').

However, the methods used to prove Lemma~\ref{la-ex5} show with only 
very minor modifications that for a tract symmetric to some ray 
the set of points escaping in a corresponding ``half-tract'' also has Hausdorff dimension~$2$.
We omit the details.
In other words, Lemma~\ref{la-ex4} together with the arguments 
in the proof of Lemma~\ref{la-ex5} show that~\eqref{ve1} and thus~\eqref{b8} hold.
Thus $f_{-2}$ has the properties claimed in Example~\ref{ex1}.

\noindent
Walter Bergweiler\\
Mathematisches Seminar\\
Christian-Albrechts-Universit\"at zu Kiel\\
Heinrich-Hecht-Platz 6\\
24098 Kiel\\
Germany

\medskip

\noindent
Jie Ding\\
School of Mathematics\\
Taiyuan University of Technology\\
Taiyuan 030024\\
China

\end{document}